\documentclass[10pt]{amsart}
      \usepackage{amsmath,amsfonts}
\def\rg{\hbox to 30pt{\rightarrowfill}}
\def\lg{\hbox to 30pt{\leftarrowfill}}

      \parskip\smallskipamount
          \newtheorem{theorem}{Theorem}[section]
      \newtheorem{definition}[theorem]{Definition}
      \newtheorem{proposition}[theorem]{Proposition}
      \newtheorem{corollary}[theorem]{Corollary}
      \newtheorem{lemma}[theorem]{Lemma}

      \makeatletter
      \@addtoreset{equation}{section}
      \makeatother

      \newcommand{\CC}{{\mathbb C}}
      \newcommand{\NN}{{\mathbb N}}
      
      \newcommand{\ZZ}{{\mathbb Z}}
      \newcommand{\DD}{{\mathbb D}}
      \newcommand{\RR}{{\mathbb R}}
      \newcommand{\FF}{{\mathbb F}}
      \newcommand{\TT}{{\mathbb T}}

      \newcommand{\cD}{{\mathcal D}}
      \newcommand{\cE}{{\mathcal E}}
      
      \newcommand{\cG}{{\mathcal G}}
      \newcommand{\cH}{{\mathcal H}}
      
      \newcommand{\cL}{{\mathcal L}}
      \newcommand{\cM}{{\mathcal M}}
      \newcommand{\cN}{{\mathcal N}}
      
      \newcommand{\cP}{{\mathcal P}}
      \newcommand{\cR}{{\mathcal R}}

      \newcommand{\rank}{\hbox{\rm{rank}}\,}

      \newdimen\expt
      \def\boxit#1{\setbox0\hbox{$\displaystyle{#1}$}
            \hbox{\lower.4\expt
       \hbox{\lower3\expt\hbox{\lower\dp0
            \hbox{\vbox{\hrule height.4\expt
       \hbox{\vrule width.4\expt\hskip3\expt
            \vbox{\vskip3\expt\box0\vskip2\expt}%
       \hskip3\expt\vrule width.4\expt}\hrule height.4\expt}}}}}}
      
\begin{document}
       \pagestyle{myheadings}
      \markboth{ Gelu Popescu}{  Euler characteristic on noncommutative  polyballs  }

      \title [   Euler characteristic on noncommutative   polyballs  ]
      {  Euler characteristic on noncommutative   polyballs }
        \author{Gelu Popescu}
\date{July 9, 2014 (revised version)}
      \thanks{Research supported in part by an NSF grant}
      \subjclass[2000]{Primary:  46L52;  32A70;  Secondary: 47A13; 47A15}
      \keywords{Noncommutative polyball;  Euler characteristic; Curvature invariant;  Berezin transform;
      Fock space; Creation operators; Invariant subspaces.
}

      \address{Department of Mathematics, The University of Texas
      at San Antonio \\ San Antonio, TX 78249, USA}
      \email{\tt gelu.popescu@utsa.edu}

\begin{abstract}
In  this paper we introduce  and study the Euler characteristic (denoted by $\chi$)   associated with  algebraic modules generated by  arbitrary elements of certain noncommutative polyballs.  We provide several asymptotic formulas for  $\chi$   and prove some of its basic properties. We show  that the Euler characteristic is a complete unitary invariant  for the finite rank Beurling type invariant subspaces of the tensor product of full Fock spaces $F^2(H_{n_1})\otimes \cdots \otimes F^2(H_{n_k})$, and  prove that  its range   coincides with the interval $[0,\infty)$.
 We obtain an analogue of Arveson's version of  the Gauss-Bonnet-Chern  theorem from Riemannian  geometry, which connects the curvature to the Euler characteristic. In particular, we prove that if $\cM$ is an  invariant subspace of  $F^2(H_{n_1})\otimes \cdots \otimes F^2(H_{n_k})$, $n_i\geq 2$, which is graded  (generated by multi-homogeneous polynomials),  then the curvature and the Euler characteristic of the orthocomplement  of $\cM$ coincide.
\end{abstract}

      \maketitle








\bigskip

\section*{Introduction}

Arveson  (\cite{Arv1}, \cite{Arv2}, \cite{Arv3}) introduced and studied the  curvature invariant and the Euler characteristic for commuting $n$-tuples of bounded linear operators $T=(T_1,\ldots, T_n)$ acting on a Hilbert space $\cH$ such that $T_1T_1^*+\cdots +T_nT_n^*\leq I$ and the defect operator $\Delta_T:=(I-T_1T_1^*-\cdots -T_nT_n^*)^{1/2}$ has finite rank. Based on certain asymptotic formulas for both the curvature and Euler characteristic, he  obtained, as the main result of the paper,  an operator-theoretic version of the Gauss-Bonnet-Chern formula of Riemannian geometry. More precisely, he proved that  for any graded (generated by homogeneous polynomials) invariant subspace $\cM$ of a finite direct sum of copies of  the symmetric Fock space $F_s^2(H_n)$  with $n$ generators, the curvature  and the Euler characteristic of the othocomplement  of  $\cM$ coincide. This result was generalized by Fang (\cite{Fang1}) to invariant subspaces generated by arbitrary polynomials. Fang also obtained a version of  the Gauss-Bonnet-Chern formula for the Hardy space  $H^2(\DD^n)$  over the polydisc(\cite{Fang}) and for the Dirichlet space over the unit disc (\cite{Fang2}, \cite{R}). The theory of Arveson's curvature on the symmetric Fock space   was significantly expanded  due to the work  by Greene, Richter, and Sundberg  \cite{GRS},  Fang  \cite{Fang1}, and   Gleason, Richter, and Sundberg  \cite{GlRS}.

In the noncommutative setting,  the author (\cite{Po-curvature}, \cite{Po-varieties}) and, independently,
Kribs \cite{Kr} defined and studied the curvature and Euler characteristic  for
arbitrary  elements $T$ with
  rank\,$\Delta_T<\infty$ in the unit ball
 $$
 [B(\cH)^{n}]_1^-:=\left\{
(X_{1},\ldots, X_{n})\in B(\cH)^{n}:\ X_{1} X_{1}^*+\cdots +X_{n}X_{n}^*\leq  I
 \right\}
$$
and, in particular,  for the invariant (resp.~coinvariant) subspaces of the  full Fock space  $F^2(H_n)$ with $n$ generators. Some of these results were extended by  Muhly and Solel \cite{MuSo4} to a class of completely positive maps on semifinite factors.
  A noncommutative analogue  (for the full Fock space)  of the  Arveson's version of  the Gauss-Bonnet-Chern  theorem was obtained in \cite{Po-curvature}.

In \cite{Po-curvature-polyballs},  we  developed a theory of curvature invariant and  multiplicity invariant
  for  tensor products of full
Fock spaces $F^2(H_{n_1})\otimes \cdots \otimes F^2(H_{n_k})$  and
also for   tensor products  of   symmetric Fock spaces
$F_s^2(H_{n_1})\otimes \cdots \otimes F_s^2(H_{n_k})$.  To prove the existence of the curvature and its basic properties in these  settings   required a new approach  based on noncommutative Berezin transforms and  multivariable operator theory on polyballs and varieties (see \cite{Po-poisson}, \cite{Po-automorphism}, \cite{Po-Berezin-poly}, and \cite{Po-Berezin3}), and also certain summability results for completely positive maps which are trace contractive.
In particular, we obtained new proofs for the existence  of the curvature on  the full Fock space $F^2(H_{n})$,   the Hardy space $H^2(\DD^k)$ (which corresponds
    to $n_1=\cdots=n_k=1$), and the symmetric Fock space $F_s^2(H_{n})$.

In  the present paper, which  is a continuation of \cite{Po-curvature-polyballs},  we introduce  and study the Euler characteristic  associated with the elements of a class of noncommutative polyballs, and obtain an analogue of Arveson's version of  the Gauss-Bonnet-Chern  theorem from Riemannian  geometry, which connects the curvature to the Euler characteristic of some associated  algebraic modules.

To present our results we need a few preliminaries.
Throughout this paper, we denote by $B(\cH)$ the algebra of bounded
linear operators on a Hilbert space $\cH$.
Let $B(\cH)^{n_1}\times_c\cdots \times_c B(\cH)^{n_k}$ be
   the set of all tuples  ${\bf X}:=({ X}_1,\ldots, { X}_k)$ in $B(\cH)^{n_1}\times\cdots \times B(\cH)^{n_k}$
     with the property that the entries of ${X}_s:=(X_{s,1},\ldots, X_{s,n_s})$  are commuting with the entries of
      ${X}_t:=(X_{t,1},\ldots, X_{t,n_t})$  for any $s,t\in \{1,\ldots, k\}$, $s\neq t$.
  Note that, for each $i\in \{1,\ldots,k\}$,  the operators $X_{i,1},\ldots, X_{i,n_i}$ are not necessarily commuting.
  Let ${\bf n}:=(n_1,\ldots, n_k)\in \NN^k$, where $n_i \in\NN:=\{1,2,\ldots\}$ and $i\in \{1,\ldots, k\}$, and define  the {\it regular polyball}
$$
{\bf B_n}(\cH):=\left\{ {\bf X}=(X_1,\ldots, X_k)\in B(\cH)^{n_1}\times_c\cdots \times_c B(\cH)^{n_k}: \ {\bf \Delta_{X}^p}(I)\geq 0 \ \text{ for }\ {\bf 0}\leq {\bf p}\leq (1,\ldots,1)\right\},
$$
where
 the {\it defect mapping} ${\bf \Delta_{X}^p}:B(\cH)\to  B(\cH)$ is defined by
$$
{\bf \Delta_{X}^p}:=\left(id -\Phi_{X_1}\right)^{p_1}\circ \cdots \circ\left(id -\Phi_{ X_k}\right)^{p_k}, \qquad p_i\in\{0,1\},
$$
 and
$\Phi_{X_i}:B(\cH)\to B(\cH)$  is the completely positive linear map defined by   $\Phi_{X_i}(Y):=\sum_{j=1}^{n_i}   X_{i,j} Y X_{i,j} ^*$ for $Y\in B(\cH)$.  We use the convention that $(id-\Phi_{f_i,X_i})^0=id$.
 For information on completely bounded (resp. positive) maps we refer to \cite{Pa-book}.

Let $H_{n_i}$ be
an $n_i$-dimensional complex  Hilbert space.
  We consider the full Fock space  of $H_{n_i}$ defined by
$$F^2(H_{n_i}):=\CC1\oplus \bigoplus_{p\geq 1} H_{n_i}^{\otimes p},$$
where $H_{n_i}^{\otimes p}$ is the
(Hilbert) tensor product of $p$ copies of $H_{n_i}$. Let $\FF_{n_i}^+$ be the unital free semigroup on $n_i$ generators
$g_{1}^i,\ldots, g_{n_i}^i$ and the identity $g_{0}^i$.
  Set $e_\alpha^i :=
e^i_{j_1}\otimes \cdots \otimes e^i_{j_p}$ if
$\alpha=g^i_{j_1}\cdots g^i_{j_p}\in \FF_{n_i}^+$
 and $e^i_{g^i_0}:= 1\in \CC$.
Note that $\{e^i_\alpha:\alpha\in\FF_{n_i}^+\}$ is an orthonormal
basis of $F^2(H_{n_i})$.  The length of $\alpha\in
\FF_{n_i}^+$ is defined by $|\alpha|:=0$ if $\alpha=g_0^i$  and
$|\alpha|:=p$ if
 $\alpha=g_{j_1}^i\cdots g_{j_p}^i$, where $j_1,\ldots, j_p\in \{1,\ldots, n_i\}$.

 For each $  n_i\in \NN$, $i\in\{1,\ldots,k\}$, and  $j\in\{1,\ldots,n_i\}$, let  $S_{i,j}:F^2(H_{n_i})\to
F^2(H_{n_i})$   be
 the {\it left creation  operator}  defined  by
$
S_{i,j} e_\alpha^i:=  e^i_{g_j \alpha}$, $ \alpha\in \FF_{n_i}^+$,
 and define
 the operator ${\bf S}_{i,j}$ acting on the tensor Hilbert space
$F^2(H_{n_1})\otimes\cdots\otimes F^2(H_{n_k})$ by setting
$${\bf S}_{i,j}:=\underbrace{I\otimes\cdots\otimes I}_{\text{${i-1}$
times}}\otimes S_{i,j}\otimes \underbrace{I\otimes\cdots\otimes
I}_{\text{${k-i}$ times}}.
$$
 The $k$-tuple ${\bf S}:=({\bf S}_1,\ldots, {\bf S}_k)$, where  ${\bf S}_i:=({\bf S}_{i,1},\ldots,{\bf S}_{i,n_i})$, is  an element  in the
regular polyball $ {\bf B_n}(\otimes_{i=1}^kF^2(H_{n_i}))$ and it
  is  called  {\it universal model} associated
  with the abstract
  polyball ${\bf B_n}:=\{{\bf B_n}(\cH):\ \cH \text{ is a Hilbert space}\}$. The noncommutative Hardy algebra $F^\infty({\bf B_n})$ is the sequential SOT-(resp. WOT-, $w^*$-) closure of all polynomials in ${\bf S}_{i,j}$ and the identity.

Let  ${\bf T}=({ T}_1,\ldots, { T}_k)\in {\bf B_n}(\cH)$ with $T_i:=(T_{i,1},\ldots, T_{i,n_i})$.
We  use the notation $T_{i,\alpha_i}:=T_{i,j_1}\cdots T_{i,j_p}$
  if $\alpha_i\in \FF_{n_i}^+$ and $\alpha_i=g_{j_1}^i\cdots g_{j_p}^i$, and
   $T_{i,g_0^i}:=I$.
The {\it noncommutative Berezin kernel} associated with any element
   ${\bf T}$ in the noncommutative polyball ${\bf B_n}(\cH)$ is the operator
   $${\bf K_{T}}: \cH \to F^2(H_{n_1})\otimes \cdots \otimes  F^2(H_{n_k}) \otimes  \overline{{\bf \Delta_{T}}(I) (\cH)}$$
   defined by
   $$
   {\bf K_{T}}h:=\sum_{\beta_i\in \FF_{n_i}^+, i=1,\ldots,k}
   e^1_{\beta_1}\otimes \cdots \otimes  e^k_{\beta_k}\otimes {\bf \Delta_{T}}(I)^{1/2} T_{1,\beta_1}^*\cdots T_{k,\beta_k}^*h,\qquad h\in \cH,
   $$
where the defect operator is defined by
$
{\bf \Delta_{T}}(I)  :=(id-\Phi_{T_1})\cdots (id-\Phi_{T_k})(I).
$
The noncommutative Berezin kernel  ${\bf K_{T}}$ is a contraction  and
    ${\bf K_{T}} { T}^*_{i,j}= ({\bf S}_{i,j}^*\otimes I)  {\bf K_{T}}
    $
    for any $i\in \{1,\ldots, k\}$ and $j\in \{1,\ldots, n_i\}$.
More about the theory of noncommutative Berezin kernels on polybals and polydomains can be found in \cite{Po-poisson}, \cite{Po-domains}, \cite{Po-automorphism}, \cite{Po-Berezin-poly},  and \cite{Po-Berezin3}.

  In a recent paper \cite{Po-curvature-polyballs}, we  introduced the {\it curvature}   of any   element  ${\bf T}\in  {\bf B_n}(\cH)$  with trace class defect, i.e. $\text{\rm trace\,}[{\bf \Delta_{T}}(I)]<\infty$,  by setting
\begin{equation*}
\text{\rm curv}\,({\bf T}):=
\lim_{m\to\infty}\frac{1}{\left(\begin{matrix} m+k\\ k\end{matrix}\right)}\sum_{{q_1\geq 0,\ldots, q_k\geq 0}\atop {q_1+\cdots +q_k\leq m}} \frac{\text{\rm trace\,}\left[ {\bf K_{T}^*} (P_{q_1}^{(1)}\otimes \cdots \otimes P_{q_k}^{(k)}\otimes I_\cH){\bf K_{T}}\right]}{\text{\rm trace\,}\left[P_{q_1}^{(1)}\otimes \cdots \otimes P_{q_k}^{(k)}\right]},
\end{equation*}
where ${\bf K_T}$ is the Berezin kernel of ${\bf T}$ and  $P_{q_i}^{(i)}$ is the orthogonal projection of the full Fock space $F^2(H_{n_i})$ onto the span of all vectors $e_{\alpha_i}^i$ with  $\alpha\in  \FF_{n_i}^+$ and $|\alpha_i|=q_i$.
The curvature is a unitary invariant  that measures how far  ${\bf T}$ is  from being  ``free'', i.e. a  multiple of the universal model ${\bf S}$. We proved the existence of the curvature and  established many asymptotic formulas and its basic properties. These results  were used to develop a theory of curvature  (resp.~multiplicity) invariant
  for  tensor products of full
Fock spaces $F^2(H_{n_1})\otimes \cdots \otimes F^2(H_{n_k})$  and
also for   tensor products  of   symmetric Fock spaces
$F_s^2(H_{n_1})\otimes \cdots \otimes F_s^2(H_{n_k})$.

Throughout the present paper, unless otherwise specified, we assume that ${\bf n}:=(n_1,\ldots, n_k)\in \NN^k$ with $n_i\geq 2$. Let ${\bf T}=({ T}_1,\ldots, {T}_k)\in B(\cH)^{n_1}\times_c\cdots \times_c B(\cH)^{n_k}$ be  a $k$-tuple such that  its defect operator
${\bf \Delta_{T}}(I)$ has  finite rank.
For each ${\bf q}=(q_1,\ldots, q_k)\in \ZZ_+^k$, where $\ZZ_+:=\{0\}\cup \NN$, we define the linear manifold
$$
M_{\bf q}({\bf T}):=\text{\rm span}\,\left\{T_{1,\alpha_1}\cdots T_{k,\alpha_k}h: \  \alpha_i\in \FF_{n_i}^+, |\alpha_i|\leq q_i,  h\in {\bf \Delta_{T}}(I) (\cH)\right\}.
$$
In Section 1, we introduce the {\it Euler characteristic }   of ${\bf T}$ by setting
\begin{equation*}
\chi({\bf T}):=
 \lim_{m\to\infty}  \frac{\dim M_{{\bf q}^{(m)}}({\bf T})}{\prod_{i=1}^k (1+ n_i+\cdots  + n_i^{q_i^{(m)}})},
\end{equation*}
where   ${\bf q}^{(m)}=({q}^{(m)}_1,\ldots, { q}^{(m)}_k)$ is a  cofinal sequence in $\ZZ_+^k$. We show the Euler characteristic exists and provide several asymptotic formulas including the following:
\begin{equation*}\begin{split}
\chi({\bf T})&=
 \lim_{{\bf q}=(q_1,\ldots, q_k)\in \ZZ_+^k}
\frac{\rank\left[ {\bf K_{T}^*}(P_{\leq(q_1,\ldots, q_k)}  \otimes I){\bf K_{T}}\right]}{\rank\left[P_{\leq(q_1,\ldots, q_k)}\right]}\\
&= \lim_{{\bf q}=(q_1,\ldots, q_k)\in \ZZ_+^k}\frac{\rank \left[(id-\Phi_{T_1}^{q_1+1})\circ \cdots \circ (id-\Phi_{T_k}^{q_k+1})(I)\right]}
{\prod_{i=1}^k(1+n_i+\cdots + n_i^{q_i})},
\end{split}
\end{equation*}
where ${\bf K_T}$ is the Berezin kernel   and $\Phi_{T_1},\ldots, \Phi_{T_k}$ are the  completely positive maps associated with ${\bf T}$. It turns out that $$
0\leq \text{\rm curv}({\bf T})\leq  \chi({\bf T})\leq \rank [{\bf \Delta_{T}}(I)].
$$
We remark that the condition $n_i\geq 2$ for all $i\in \{1,\ldots, k\}$ is needed to prove the existence of the Euler characteristic (see  the proof of Theorem \ref{Eul}). However, there are indications that the result remains true if one ore more of the $n_i$'s assume the value one. For now, this remains an open problem. We should add the fact that the curvature invariant exists  if $n_i\geq 1$ (see \cite{Po-curvature-polyballs}). On the other hand, an important question that remains open is whether one can drop the condition that $T_{i,j}$ commutes with $T_{s,t}$ for $i\neq s$.

We say that $\cM$
 is an invariant subspace of the tensor product $F^2(H_{n_1})\otimes \cdots \otimes F^2(H_{n_k})\otimes \cH$  or  that $\cM$  is    invariant  under  ${\bf S}\otimes I_\cH$ if it is invariant under each operator
  ${\bf S}_{i,j}\otimes I_\cH$. When $({\bf S}\otimes I_\cH)|_\cM$ has finite rank defect, we say that $\cM$ has finite rank.
Given two invariant subspaces $\cM$ and $\cN$ under ${\bf S}\otimes I_\cH$,  we say that they are unitarily equivalent if there is a unitary operator $U:\cM\to \cN$ such that $U({\bf S}_{i,j}\otimes I_\cH)|_\cM=({\bf S}_{i,j}\otimes I_\cH)|_\cN U$.

In Section 2,  we present a more direct and more transparent proof for the characterization of the  invariant subspaces of $\otimes_{i=1}^kF^2(H_{n_i})$ with positive defect operator, i.e $\Delta_\cM:={\bf \Delta_{S\otimes {\it I}_\cH}}(P_\cM)\geq 0$.  This result complements the corresponding one from \cite{Po-Berezin-poly} and   brings new insight concerning the structure of the invariant subspaces of the tensor product $\otimes_{i=1}^kF^2(H_{n_i})$. The invariant subspaces with positive defect operators are the so-called {\it Beurling type invariant subspaces}.

We show  that  the Euler characteristic completely classifies the
 finite rank  Beurling type invariant subspaces  of ${\bf S}\otimes I_\cH$  which do not contain reducing subspaces. In particular, the Euler characteristic classifies  the  finite rank  Beurling type invariant subspaces  of $ F^2(H_{n_1})\otimes \cdots \otimes  F^2(H_{n_k})$ (see Theorem \ref{classification}).

Let $\cM$ be an invariant subspace of the tensor product $F^2(H_{n_1})\otimes\cdots\otimes F^2(H_{n_k})\otimes \cE$, where $\cE$ is a finite dimensional Hilbert space. We introduce  the Euler characteristic   of  the orthocomplement $\cM^\perp$ by setting
\begin{equation*}
\begin{split}
\chi(\cM^\perp)&:=\lim_{q_1\to\infty}\cdots\lim_{q_k\to\infty}
\frac{\rank\left[P_{\cM^\perp}(P_{\leq(q_1,\ldots, q_k)}\otimes I_\cE)  \right]}{\rank\left[P_{\leq(q_1,\ldots, q_k)}\right]},
\end{split}
\end{equation*}
where $P_{\leq(q_1,\ldots, q_k)}$  is the orthogonal projection of $\otimes_{i=1}^k F^2(H_{n_i})$ onto the  span of all vectors of the form $e_{\alpha_1}^1\otimes\cdots \otimes e_{\alpha_k}^k$, where $\alpha_i\in \FF_{n_i}^+$, $|\alpha_i|\leq q_i$.
In Section 2, we show that the Euler characteristic of  $\cM^\perp$ exists and satisfies the equation
 \begin{equation*}
\begin{split}
\chi(\cM^\perp)=\chi({\bf M}),
\end{split}
\end{equation*}
where  ${\bf M}$ is the compression of ${\bf S}\otimes I_\cE$ to the orthocomplement of $\cM$.  This is used to show that for any $t\in [0,1]$ there is an invariant subspace $\cM(t)$ of $F^2(H_{n_1})\otimes\cdots\otimes F^2(H_{n_k})$ such that $\chi(\cM(t)^\perp)=t$ and, consequently, the range of the Euler characteristic coincides with the interval $[0,\infty)$. If $k=1$, the corresponding result was proved in \cite{Po-curvature} and \cite{Kr}.
Moreover, if $k\geq 2$,  we prove that  for each $t\in(0,1)$,  there exists an uncountable family $\{T^{(\omega)}(t)\}_{\omega\in \Omega}$ of  non-isomorphic pure elements  of rank  one defect  in the regular polyball  such that
$$
\chi(T^{(\omega)}(t))=t, \qquad  \text{ for all } \omega\in \Omega.
$$

 In Section 3,  we provide a characterization of the graded invariant subspaces of  $F^2(H_{n_1})\otimes\cdots\otimes F^2(H_{n_k})$ with positive defect operator $\Delta_\cM$.
We  also prove that if $\cM$ is any graded  invariant subspace of the tensor product $F^2(H_{n_1})\otimes\cdots\otimes F^2(H_{n_k})$, then
 \begin{equation*}
\begin{split}
\lim_{q_1\to\infty}\cdots\lim_{q_k\to\infty}
\frac{\rank\left[P_{\cM^\perp}P_{\leq(q_1,\ldots, q_k)}  \right]}{\rank\left[P_{\leq(q_1,\ldots, q_k)}\right]}
=\lim_{q_1\to\infty}\cdots\lim_{q_k\to\infty}
\frac{\text{\rm trace}\,\left[P_{\cM^\perp}P_{\leq(q_1,\ldots, q_k)}  \right]}{\text{\rm trace}\, \left[P_{\leq(q_1,\ldots, q_k)}\right]}.
\end{split}
\end{equation*}
This is equivalent to
$$\chi(P_{\cM^\perp}{\bf S}|_{\cM^\perp})
=\text{\rm curv}\,(P_{\cM^\perp}{\bf S}|_{\cM^\perp}),
$$
which is  an analogue of Arveson's version \cite{Arv2}  of the  Gauss-Bonnet-Chern theorem from Riemannian geometry in our setting.
We recall
\cite{Po-Berezin-poly} that any  pure element with rank-one defect in the polyball ${\bf B}_{\bf n}(\cH)$ has the form $P_{\cM^\perp}{\bf S}|_{\cM^\perp}$, where
$\cM$ is an  invariant subspace of the tensor product $F^2(H_{n_1})\otimes\cdots\otimes F^2(H_{n_k})$. It remains an open problem whether, as in the commutative setting (\cite{Fang1}, \cite{Fang2}), the result above holds true for any invariant subspace $\cM$ generated by arbitrary polynomials.  We mention that  a version of the Gauss-Bonnet-Chern theorem for graded pure elements of finite rank  defect in the noncommutative polyball is also obtained.
Finally, we remark that  the results of this paper  can be  re-formulated in terms of Hilbert modules \cite{DoPa} over  the complex semigroup algebra $\CC[\FF_{n_1}^+\times \cdots \times \FF_{n_k}^+]$ generated by  direct product of  free semigroups. In this setting, the Hilbert module associated with the universal model  of the abstract polyball ${\bf B_n}$ plays the role of rank-one free module in the algebraic theory \cite{K}.

\bigskip

 \section{Euler characteristic on noncommutative polyballs}

In this section we prove the existence of the Euler characteristic of any element ${\bf T}$ in the noncommutative polyball and provide several asymptotic formulas  in terms of the noncommutative Berezin kernel and the completely positive maps associated with ${\bf T}\in {\bf B_n}(\cH)$. Basic properties of the Euler characteristic are also proved.

Let  ${\bf T}=({ T}_1,\ldots, { T}_k)\in B(\cH)^{n_1}\times_c\cdots \times_c B(\cH)^{n_k} $ with $T_i:=(T_{i,1},\ldots, T_{i,n_i})$, and let $\cD\subset \cH$ be a finite dimensional subspace.  For each ${\bf q}=(q_1,\ldots, q_k)\in \ZZ_+^k$, where $\ZZ_+=\{0,1,\ldots\}$, we define
$$
M_{\bf q}({\bf T},\cD):=\text{\rm span}\,\left\{T_{1,\alpha_1}\cdots T_{k,\alpha_k}h: \  \alpha_i\in \FF_{n_i}^+, |\alpha_i|\leq q_i,  h\in \cD\right\}.
$$
We also use the notation $M_{(q_1,\ldots, q_k)}$ when ${\bf T}$ and  $\cD$ are  understood and we want to emphasis the coordinates $q_1,\ldots q_k$.
Given two $k$-tuples ${\bf q}=(q_1,\ldots, q_k)$ and ${\bf p}=(p_1,\ldots, p_k)$ in $\ZZ_+^k$, we consider the partial order ${\bf q}\leq {\bf p}$ defined by $q_i\leq p_i$ for any $i\in \{1,\ldots, k\}$. We consider $\ZZ_+^k$ as a directed set with respect to this partial order.

\begin{theorem} \label{Eul} Let $\cD\subset \cH$ be a finite dimensional subspace and let  ${\bf T}\in B(\cH)^{n_1}\times_c\cdots \times_c B(\cH)^{n_k}$, where
$n_i\geq 2$ for any $i\in \{1,\ldots,k\}$. Then the following iterated limits exist and are  equal
$$
 \lim_{q_{\sigma(1)}\to\infty}\cdots\lim_{q_{\sigma(k)}\to\infty} \frac{\dim M_{{\bf q}}({\bf T},{\cD})}{n_1^{q_1}\cdots n_k^{q_k}},
$$
where  $\sigma$  is any  permutation of $\{1,\ldots, k\}$ and ${\bf q}=(q_1,\ldots, q_k)\in \ZZ_+^k$. Moreover, these limits coincide with
$$
\lim_{{\bf q}=(q_1,\ldots, q_k)\in \ZZ_+^k} \frac{\dim M_{{\bf q}}({\bf T},{\cD})}{n_1^{q_1}\cdots n_k^{q_k}}.
$$
\end{theorem}
\begin{proof} If ${\bf q}=(q_1,\ldots, q_k)\in \ZZ_+^k$, we use the notation
$M_{(q_1,\ldots, q_k)}:=M_{{\bf q}}({\bf T},{\cD})$.
For  $i\in \{1,\ldots, k\}$,  note that
\begin{equation*}
M_{(q_1,\ldots, q_k)}=M_{(q_1,\ldots q_{i-1}, 0, q_{i+1}\ldots, q_k)} +
T_{i,1}M_{(q_1,\ldots q_{i-1}, q_i-1, q_{i+1}\ldots, q_k)}+\cdots + T_{i,n_i}M_{(q_1,\ldots q_{i-1}, q_i-1, q_{i+1}\ldots, q_k)}.
\end{equation*}
Consequently, we deduce that

\begin{equation} \label{ni}
\dim M_{(q_1,\ldots, q_k)}\leq  n_i \dim M_{(q_1,\ldots q_{i-1}, q_i-1, q_{i+1}\ldots, q_k)} +\dim M_{(q_1,\ldots q_{i-1}, 0, q_{i+1}\ldots, q_k)}
 \end{equation}
for any $(q_1,\ldots, q_k)\in \ZZ_+^k$. Iterating this inequality, we obtain
\begin{equation} \label{ni2}
\dim M_{(q_1,\ldots, q_k)}\leq (1+n_i+\cdots + n_i^{q_i}) \dim M_{(q_1,\ldots q_{i-1}, 0, q_{i+1}\ldots, q_k)},
\end{equation}
which implies
\begin{equation} \label{ni3}
\frac{\dim M_{(q_1,\ldots, q_k)}}{\prod_{i=1}^k (1+n_i+\cdots + n_i^{q_i}) }
\leq \dim M_{(0,\ldots, 0)}=\dim \cD
\end{equation}
for any $(q_1,\ldots, q_k)\in \ZZ_+^k$.  On the other hand, relations \eqref{ni}, \eqref{ni2}, and \eqref{ni3}  imply
\begin{equation}\label{ni4}
0\leq \frac{\dim M_{(q_1,\ldots, q_k)}}{n_i^{q_i}}\leq  \frac{ \dim M_{(q_1,\ldots q_{i-1}, q_i-1, q_{i+1}\ldots, q_k)}}{n_i^{q_i-1}}+\frac{1}{n_i^{q_i}}\prod_{{s=1}\atop{s\neq i}}^k(1+n_s+\cdots + n_s^{q_s})\dim \cD.
\end{equation}
For $x\in \RR$ we set $x^+:=\max \{x,0\}$ and $x^-:=\max \{-x,0\}$. Denote $x_{q_i}:=\frac{\dim M_{(q_1,\ldots, q_k)}}{n_i^{q_i}}$ and note that
$$
0\leq x_N=x_0+\sum_{p=1}^N (x_p-x_{p-1})^+-\sum_{p=1}^N(x_p-x_{p-1})^-.
$$
Hence, using relation \eqref{ni4}, we deduce that
$$
\sum_{p=1}^N(x_p-x_{p-1})^-\leq x_0+\sum_{p=1}^N (x_p-x_{p-1})^+
\leq x_0+\left(\sum_{p=1}^N\frac{1}{n_i^p}\right) \prod_{{s=1}\atop{s\neq i}}^k(1+n_s+\cdots + n_s^{q_s})\dim \cD.
$$
Consequently,   we obtain
$$
x_0+ \sum_{p=1}^N |x_p-x_{p-1}|\leq 2x_0+2\sum_{p=1}^N (x_p-x_{p-1})^+
$$
for any $N\in \NN$.
Since $n_i\geq 2$, the inequalities above show that  the sequence
$x_N=x_0+\sum_{p=1}^N (x_p-x_{p-1})$ is convergent as $N\to \infty$. Therefore,
\begin{equation}\label{LIM}
\lim_{q_i\to\infty}\frac{\dim M_{(q_1,\ldots, q_k)}}{n_i^{q_i}}\quad \text{  exists for any } \quad q_1,\ldots q_{i-1},  q_{i+1}\ldots, q_k\in \ZZ_+.
\end{equation}

The next step in our proof is to show that the iterated limit
$$
 \lim_{q_{k}\to\infty}\cdots\lim_{q_{1}\to\infty} \frac{\dim M_{(q_1,\ldots, q_k)}}{n_1^{q_1}\cdots n_k^{q_k}} \quad \text{ exists.}
$$
We use an inductive argument. Due to  relation \eqref{LIM},  the limit
$\lim_{q_1\to\infty}\frac{\dim M_{(q_1,\ldots, q_k)}}{n_1^{q_1}}$  exists for any  $q_2, \ldots, q_k$ in $\ZZ_+$. Assume that $1\leq p\leq k-1$ and that  the iterated limit
$$
 y(q_{p+1},\ldots, q_k):=\lim_{q_{p}\to\infty}\cdots\lim_{q_{1}\to\infty} \frac{\dim M_{(q_1,\ldots, q_k)}}{n_1^{q_1}\cdots n_p^{q_p}}
 $$  exists  for any  $q_{p+1}, \ldots, q_k$ in $\ZZ_+$.
Due to relation \eqref{ni4}, we have
\begin{equation*}
0\leq \frac{\dim M_{(q_1,\ldots, q_k)}}{n_1^{q_1}\cdots n_p^{q_p}}\leq  n_{p+1}\frac{ \dim M_{(q_1,\ldots q_{p}, q_{p+1}-1, q_{p+2}\ldots, q_k)}}{n_1^{q_1}\cdots n_p^{q_p}}+\frac{1}{n_1^{q_1}\cdots n_p^{q_p}}\prod_{{s=1}\atop{s\neq p+1}}^k(1+n_s+\cdots + n_s^{q_s})\dim \cD.
\end{equation*}
Consequently, taking the limits as $q_1\to\infty$, \ldots, $q_p\to \infty$,  we deduce that
$$
\frac{y(q_{p+1},\ldots, q_k)} {n_{p+1}^{q_{p+1}}}\leq
\frac{y(q_{p+1}-1,q_{p+2},\ldots, q_k)} {n_{p+1}^{q_{p+1}-1}}
+\frac{1}{n_{p+1}^{q_{p+1}}}\prod_{s=1}^p \frac{n_s}{n_s-1}
\prod_{s=p+2}^k(1+n_s+\cdots + n_s^{q_s})\dim \cD.
$$
Setting
$z_{q_{p+1}}:=\frac{y(q_{p+1},\ldots, q_k)} {n_{p+1}^{q_{p+1}}}$,
we have
$$
0\leq z_{q_{p+1}}\leq z_{q_{p+1}-1}+ \frac{1}{n_{p+1}^{q_{p+1}}}\prod_{s=1}^p \frac{n_s}{n_s-1}
\prod_{s=p+2}^k(1+n_s+\cdots + n_s^{q_s})\dim \cD, \qquad q_{p+1}\in \NN.
$$
Similarly to the proof that $x_{q_i}$ is convergent  (see relation \eqref{LIM}), one can show that the sequence $\{z_{q_{p+1}}\}$ is convergent as $q_{p+1}\to \infty$ for any $q_{p+2},\ldots, q_k$ in $\ZZ_+$. Therefore,
 the iterated limit
$$
 y(q_{p+2},\ldots, q_k):= \lim_{q_{p+1}\to\infty}\cdots\lim_{q_{1}\to\infty} \frac{\dim M_{(q_1,\ldots, q_k)}}{n_1^{q_1}\cdots n_{p+1}^{q_{p+1}}} \quad \text{ exists},
$$
which proves our assertion. Since the entries of ${X}_s:=(X_{s,1},\ldots, X_{s,n_s})$  are commuting with the entries of
      ${X}_t:=(X_{t,1},\ldots, X_{t,n_t})$  for any $s,t\in \{1,\ldots, k\}$, $s\neq t$, similar arguments as above show that, for  any  permutation $\sigma$   of $\{1,\ldots, k\}$,
      the iterated limit
$
 L_\sigma:=\lim_{q_{\sigma(1)}\to\infty}\cdots\lim_{q_{\sigma(k)}\to\infty} a_{(q_1,\ldots, q_k)}
$
exists, where $a_{(q_1,\ldots, q_k)}:=\frac{\dim M_{(q_1,\ldots, q_k)}}{n_1^{q_1}\cdots n_k^{q_k}}$.

The next step  is to show that all these iterated limits are equal. We proceed by contradiction.
Without loss of generality we can assume that
$
L_{id}
< L_\sigma
$
for some permutation $\sigma$.
Let $\epsilon>0$ be such that $\epsilon <\frac{L_\sigma-L_{id}}{4}$. Since $n_i\geq 2$ and $L_{id}=\lim_{q_{1}\to\infty}\cdots\lim_{q_{k}\to\infty}a_{(q_1,\ldots, q_k)}$,    we can  choose  natural numbers $N_1\leq N_2\leq \ldots \leq N_k$ such that
$\sum_{j=N_i+1}^\infty \frac{1}{n_i^j}<\frac{\epsilon}{kM} $ for any $i\in \{1,\ldots, k\}$, where $M:=(\dim\cD) \prod_{i=1}^k\frac{n_i}{n_i-1}$,
and such that
$$|a_{(N_1,\ldots, N_k)}-L_{id}|<\epsilon.
$$
Taking into account that $
 L_\sigma=\lim_{q_{\sigma(1)}\to\infty}\cdots\lim_{q_{\sigma(k)}\to\infty} a_{(q_1,\ldots, q_k)},
$
we can choose natural numbers $C_1,\ldots, C_k$ such that $C_i>\max\{N_1,\ldots, N_k\}$ and
$$
|a_{(C_1,\ldots, C_k)}-L_\sigma|<\epsilon.
$$
Consequently, we obtain
\begin{equation}
\label{AN}
a_{(C_1,\ldots, C_k)}-a_{(N_1,\ldots, N_k)}>2\epsilon.
\end{equation}
For each $(q_1,\ldots, q_k)\in \ZZ_+^k$, set
$$
d_{(q_1,\ldots, q_k)}:=\frac{1}{n_1^{q_1}\cdots n_k^{q_k}}
\prod_{i=1}^k (1+n_i+\cdots n_i^{q_i}) \dim\cD
$$
and note that $d_{(q_1,\ldots, q_k)}\leq (\dim\cD) \prod_{i=1}^k\frac{n_i}{n_i-1}$.
Using  relation \eqref{ni}, we deduce that
\begin{equation*}
\begin{split}
&a_{(C_1,C_2,\ldots, C_k)}\\
&=a_{(N_1,N_2,\ldots, N_k)} +[(a_{(N_1+1,N_2,\ldots, N_k)}-a_{(N_1,N_2,\ldots, N_k)})+\cdots +(a_{(C_1,N_2,\ldots, N_k)}-a_{(C_1-1,N_2,\ldots, N_k)})]\\
&\quad + [(a_{(C_1,N_2+1,N_3,\ldots, N_k)}-a_{(C_1,N_2,N_3,\ldots, N_k)})+\cdots +(a_{(C_1,C_2,N_3,\ldots, N_k)}-a_{(C_1,C_2-1,N_3\ldots, N_k)})]\\
&\quad \vdots\\
&\quad  + [(a_{(C_1,C_2,\ldots,C_{k-1}, N_k+1)}-a_{(C_1,C_2,\ldots, C_{k-1}, N_k)})+\cdots +(a_{(C_1,C_2,\ldots, C_k)}-a_{(C_1,C_2,\ldots, C_{k-1}, C_k-1)})]\\
&\leq a_{(N_1,N_2,\ldots, N_k)}
+\sum_{j=N_1+1}^{C_1} \frac{1}{n_1^j} \frac{\dim M_{(0,N_2,\ldots, N_k)}}{n_2^{N_2}\cdots n_k^{N_k}}\\
&\quad +\sum_{j=N_2+1}^{C_2} \frac{1}{n_2^j} \frac{\dim M_{(C_1,0,N_3,\ldots, N_k)}}{n_1^{C_1}n_3^{N_3}\cdots n_k^{N_k}} +\cdots
+\sum_{j=N_k+1}^{C_k} \frac{1}{n_k^j} \frac{\dim M_{(C_1,0,N_3,\ldots, N_k)}}{n_1^{C_1}n_2^{N_2}\cdots n_{k-1}^{N_{k-1}}}\\
&\leq
a_{(N_1,N_2,\ldots, N_k)} + d_{(0,N_2,\ldots, N_k)}
\left(\sum_{j=N_1+1}^{C_1} \frac{1}{n_1^j}\right)\\
&\quad +
d_{(C_1,0,N_3,\ldots, N_k)}
\left(\sum_{j=N_2+1}^{C_2} \frac{1}{n_2^j}\right)+\cdots +
d_{(C_1,\ldots,C_{k-1},0)}
\left(\sum_{j=N_k+1}^{C_k} \frac{1}{n_1^j}\right)\\
&\leq
a_{(N_1,N_2,\ldots, N_k)} + \frac{\epsilon}{kM}  k (\dim\cD) \prod_{i=1}^k\frac{n_i}{n_i-1}\\
&=a_{(N_1,N_2,\ldots, N_k)}  +\epsilon.
\end{split}
\end{equation*}
Consequently, we have
 $$0\leq a_{(C_1,C_2,\ldots, C_k)} -a_{(N_1,N_2,\ldots, N_k)}  \leq \epsilon,
 $$
which contradicts relation \eqref{AN}. Therefore, we must have $L_{id}=L_\sigma$. This completes the proof that the  iterated limits $\lim_{q_{\sigma(1)}\to\infty}\cdots\lim_{q_{\sigma(k)}\to\infty} a_{(q_1,\ldots, q_k)}$ exist and are equal for any permutation $\sigma$ of the set $\{1,\ldots, k\}$.

Now, we show that the net $\{a_{(q_1,\ldots, q_k)}\}_{(q_1,\ldots, q_k)\in \ZZ_+^k}$ is convergent and
$$
\lim_{(q_1,\ldots, q_k)\in \ZZ_+^k} a_{(q_1,\ldots, q_k)}=L_{id}.
$$
Assume, by contradiction, that the net $\{a_{(q_1,\ldots, q_k)}\}_{(q_1,\ldots, q_k)\in \ZZ_+^k}$ is not  convergent to $L_{id}$. Then we can find  $\epsilon_0>0$ such that for any ${\bf q}\in \ZZ_+^k$ there exists ${\bf p}\in \ZZ_+^k$ with ${\bf p}\geq {\bf q}$ such that
$$
a_{\bf p}\in (-\infty, L_{id}-3\epsilon_0]\cup [L_{id}+3 \epsilon_0, \infty).
$$
In particular, for each ${\bf q}^{(n)}=(n,\ldots,n)\in \ZZ_+^k$ there exists ${\bf p}^{(n)}\in \ZZ_+^k$ with ${\bf p}^{(n)}>{\bf q}^{(n)}$ and such that the relation above holds. This implies that at least one of the sets
$$
\Gamma_-:=\{{\bf q}\in \ZZ_+^k:\ a_{\bf q}\leq L-3\epsilon_0\}\quad \text{ and }\quad \Gamma_+:=\{{\bf q}\in \ZZ_+^k:\ a_{\bf q}\geq L+3\epsilon_0\}
$$
 is cofinal in $\ZZ_+^k$.  Assume that $\Gamma_-$ is cofinal in $\ZZ_+^k$.  Then, using that $n_i\geq 2$  for $i\in \{1,\ldots, k\}$, we can find natural numbers $D_1,\ldots, D_k$ such that
 $\sum_{j=D_i+1}^\infty \frac{1}{n_i^j}<\frac{\epsilon_0}{kM}$ for $i\in \{1,\ldots, k\}$, where $M:=(\dim\cD) \prod_{i=1}^k\frac{n_i}{n_i-1}$,
and such that
$$a_{(D_1,\ldots, D_k)}\leq L_{id}-3\epsilon_0.
$$
Since $L_{id}=\lim_{q_{1}\to\infty}\cdots\lim_{q_{k}\to\infty}a_{(q_1,\ldots, q_k)}$, we can choose natural numbers $R_1,\ldots, R_k$ such that $R_i>C_i$ for $i\in \{1,\ldots, k\}$ and such that
$$
a_{(R_1,\ldots, R_k)}\in (L_{id}-\epsilon_0, L+\epsilon_0).
$$
 Consequently, we deduce that $a_{(R_1,\ldots, R_k)}-a_{(D_1,\ldots, D_k)}\geq 2\epsilon_0$.
 On the other hand, as in the first part of the proof, one can show that
 $$
 a_{(R_1,\ldots, R_k)}\leq a_{(D_1,\ldots, D_k)}+\epsilon_0,
 $$
 which is a contradiction.

   Now, assume that $\Gamma_+$ is a cofinal set in $\ZZ_+^k$. Since $n_i\geq 2$ and $L_{id}=\lim_{q_{1}\to\infty}\cdots\lim_{q_{k}\to\infty}a_{(q_1,\ldots, q_k)}$,    we can  choose  natural numbers $N_1\leq N_2\leq \ldots \leq N_k$ such that
$\sum_{j=N_i+1}^\infty \frac{1}{n_i^j}<\frac{\epsilon}{kM} $ for any $i\in \{1,\ldots, k\}$, where $M:=(\dim\cD) \prod_{i=1}^k\frac{n_i}{n_i-1}$,
and such that
$|a_{(N_1,\ldots, N_k)}-L_{id}|<\epsilon.
$
  Taking into account that $\Gamma_+$ is a cofinal set in $\ZZ_+^k$, we can find natural numbers $N_1',\ldots, N_k'$ such that $N_i'\geq N_i$ for any $i\in \{1,\ldots,k\}$ and such that
$$
a_{(N_1',\ldots, N_k')}\geq L_{id}+3\epsilon_0.
$$
 Consequently,   we have   $a_{(N_1',\ldots, N_k')}-a_{(N_1,\ldots, N_k)}\geq 2\epsilon_0$.
 Once again, as in the first part of the proof,  since$(N_1',\ldots, N_k')\geq (N_1,\ldots, N_k)$ one can show that
 $a_{(N_1',\ldots, N_k')}-a_{(N_1,\ldots, N_k)}\leq\epsilon_0$, which is a contradiction.
     This completes the proof.
\end{proof}

 Let ${\bf T}=({ T}_1,\ldots, {T}_k)\in B(\cH)^{n_1}\times_c\cdots \times_c B(\cH)^{n_k}$ be  a $k$-tuple such that  its defect operator defined by
${\bf \Delta_{T}}(I):=(id-\Phi_{T_1})\cdots (id-\Phi_{T_k})(I)$ has  finite rank. In this case, we say that ${\bf T}$ has finite rank defect or that ${\bf T}$ has  finite rank.
For each ${\bf q}=(q_1,\ldots, q_k)\in \ZZ_+^k$, we define
$$
M_{\bf q}({\bf T}):=\text{\rm span}\,\left\{T_{1,\alpha_1}\cdots T_{k,\alpha_k}h: \  \alpha_i\in \FF_{n_i}^+, |\alpha_i|\leq q_i,  h\in {\bf \Delta_{T}}(I) (\cH)\right\}.
$$
Let $\{{\bf q}^{(m)}\}_{m=1}^\infty$  with ${\bf q}^{(m)}=({q}^{(m)}_1,\ldots, { q}^{(m)}_k)$ be an increasing  cofinal sequence in $\ZZ_+^k$. Then
 $$
 M_{{\bf q}^{(1)}}({\bf T})\subseteq M_{{\bf q}^{(2)}}({\bf T})\subseteq \cdots \subseteq M({\bf T}) \quad \text{ and } \quad
\bigcup_{m=1}^\infty M_{{\bf q}{(m)}}({\bf T})=M({\bf T}),
$$
where
$$
M({\bf T}):=\text{\rm span}\,\left\{T_{1,\alpha_1}\cdots T_{k,\alpha_k}h: \  \alpha_i\in \FF_{n_i}^+,   h\in {\bf \Delta_{T}}(I)^{1/2} (\cH)\right\}.
$$
We introduce the {\it Euler characteristic }   of ${\bf T}$ by setting
\begin{equation*}
\chi({\bf T}):=
 \lim_{m\to\infty}  \frac{\dim M_{{\bf q}^{(m)}}({\bf T})}{\prod_{i=1}^k (1+ n_i+\cdots  + n_i^{q_i^{(m)}})}.
\end{equation*}
  We remark that the Euler characteristic does not depend on any topology associated with the Hilbert space $\cH$ but  depends solely on the linear algebra of  the linear submanifold  $M({\bf T})$ of $\cH$.

For each $q_i\in \{0,1,\ldots\}$ and  $i\in \{1,\ldots, k\}$, let $P_{q_i}^{(i)}$ be the orthogonal projection of the full Fock space $F^2(H_{n_i})$ onto the span of all vectors $e_{\alpha_i}^i$ with  $\alpha\in  \FF_{n_i}^+$ and $|\alpha_i|=q_i$. We recall that $P_{\leq(q_1,\ldots, q_k)}$  is  the orthogonal projection of $\otimes_{i=1}^k F^2(H_{n_i})$ onto the  span of all vectors of the form $e_{\alpha_1}^1\otimes\cdots \otimes e_{\alpha_k}^k$, where $\alpha_i\in \FF_{n_i}^+$, $|\alpha_i|\leq q_i$.
In what follows,  we  prove the existence of  the Euler characteristic associated with each  element in the regular polyball and establish several  asymptotic formulas.

\begin{theorem} \label{Euler1} Let  \, ${\bf T}=({ T}_1,\ldots, { T}_k)$ be  in the regular polyball  $ {\bf B_n}(\cH)$ and have  finite rank defect. If ${\bf n}:=(n_1,\ldots, n_k)$ with $n_i\geq 2$,  then the Euler characteristic of ${\bf T}$ exists and satisfies the  asymptotic formulas
\begin{equation*}\begin{split}
\chi({\bf T})&=\lim_{{\bf q}=(q_1,\ldots, q_k)\in \ZZ_+^k} \frac{\dim M_{{\bf q}}({\bf T})}{n_1^{q_1}\cdots n_k^{q_k}} \prod_{i=1}^k \left(1-\frac{1}{n_i}\right)\\
&=
 \lim_{{\bf q}=(q_1,\ldots, q_k)\in \ZZ_+^k}
\frac{\rank\left[ {\bf K_{T}^*}(P_{\leq(q_1,\ldots, q_k)}  \otimes I){\bf K_{T}}\right]}{\rank\left[P_{\leq(q_1,\ldots, q_k)}\right]}\\
&= \lim_{{\bf q}=(q_1,\ldots, q_k)\in \ZZ_+^k}\frac{\rank \left[(id-\Phi_{T_1}^{q_1+1})\circ \cdots \circ (id-\Phi_{T_k}^{q_k+1})(I)\right]}
{\prod_{i=1}^k(1+n_i+\cdots + n_i^{q_i})},
\end{split}
\end{equation*}
where ${\bf K_T}$ is the Berezin kernel   and $\Phi_{T_1},\ldots, \Phi_{T_k}$ are the  completely positive maps associated with ${\bf T}$.
\end{theorem}
\begin{proof}
Applying Theorem \ref{Eul} when $\cD={\bf \Delta_{T}}(I)^{1/2} (\cH)$,  we deduce that $\lim_{m\to\infty}  \frac{\dim M_{{\bf q}^{(m)}}({\bf T})}{n_1^{q_1^{(m)}}\cdots n_k^{q_k^{(m)}}}$ exists and  does not depend on the choice of the cofinal  sequence $\{{\bf q}^{(m)}=(q_1^{(m)},\ldots, q_k^{(m)})\}_{m=1}^\infty$ in $\ZZ_+^k$.  Moreover, the limit coincides with
$$
\lim_{{\bf q}=(q_1,\ldots, q_k)\in \ZZ_+^k} \frac{\dim M_{{\bf q}}({\bf T})}{n_1^{q_1}\cdots n_k^{q_k}}=\lim_{q_{1}\to\infty}\cdots\lim_{q_{k}\to\infty} \frac{\dim M_{\bf q}({\bf T})}{n_1^{q_1}\cdots n_k^{q_k}}.
$$
 Consequently, the Euler characteristic of ${\bf T}$ exists and  the first  equality  in the theorem holds.
The  noncommutative Berezin kernel associated with
   ${\bf T}\in {\bf B_n}(\cH)$ is the operator
   $${\bf K_{T}}: \cH \to F^2(H_{n_1})\otimes \cdots \otimes  F^2(H_{n_k}) \otimes  \overline{{\bf \Delta_{T}}(I) (\cH)}$$
   defined by
   $$
   {\bf K_{T}}h:=\sum_{\beta_i\in \FF_{n_i}^+, i=1,\ldots,k}
   e^1_{\beta_1}\otimes \cdots \otimes  e^k_{\beta_k}\otimes {\bf \Delta_{T}}(I)^{1/2} T_{1,\beta_1}^*\cdots T_{k,\beta_k}^*h, \qquad h\in \cH.
   $$
 It is easy to see that
 $$
 {\bf K_{T}^*} (e^1_{\beta_1}\otimes \cdots \otimes  e^k_{\beta_k}\otimes  h)=T_{1,\beta_1}\cdots T_{k,\beta_k}{\bf \Delta_{T}}(I)^{1/2}h
 $$
 for any $\beta_i\in \FF_{n_i}^+$ and $ i\in \{1,\ldots,k\}$. This shows that the range of
 ${\bf K_{T}^*} (P_{\leq(q_1,\ldots, q_k)}\otimes I)$  is equal to $M_{(q_1,\ldots, q_k)}({\bf T}):=M_{{\bf q}}({\bf T})$.  Hence, the range of
$ {\bf K_{T}^*}(P_{\leq(q_1,\ldots, q_k)}  \otimes I){\bf K_{T}}$ is equal to $M_{{\bf q}}({\bf T})$. Since
$$\rank\left[P_{\leq(q_1,\ldots, q_k)}\right]=\prod_{i=1}^k(1+n_i+\cdots + n_i^{q_i}),
$$
 the results above imply
 $$
 \chi({\bf T})=\lim_{q_1\to\infty}\cdots\lim_{q_k\to\infty}
\frac{\rank\left[ {\bf K_{T}^*}(P_{\leq(q_1,\ldots, q_k)}  \otimes I){\bf K_{T}}\right]}{\rank\left[P_{\leq(q_1,\ldots, q_k)}\right]}.
$$
The next step in our proof is to show that
\begin{equation}\label{KPK}
{\bf K_{T}^*}(P_{\leq(q_1,\ldots, q_k)}  \otimes I){\bf K_{T}}=(id-\Phi_{T_1}^{q_1+1})\circ \cdots \circ (id-\Phi_{T_k}^{q_k+1})(I)
\end{equation}
for any $(q_1,\ldots, q_k)\in \ZZ_+^k$.
Let   ${\bf S}:=({\bf S}_1,\ldots, {\bf S}_k)\in {\bf B_n}(\otimes_{i=1}^kF^2(H_{n_i}))$ with  ${\bf S}_i:=({\bf S}_{i,1},\ldots,{\bf S}_{i,n_i})$  be  the  universal model associated
  with the abstract
  polyball ${\bf B_n}$.  It is easy to see that     ${\bf S}:=({\bf S}_1,\ldots, {\bf S}_k)$ is  a pure $k$-tuple, i.e. $\Phi_{{\bf S}_i}^p(I)\to 0$ strongly as $p\to\infty$,   and
   $$(id-\Phi_{{\bf S}_1})\circ \cdots \circ (id-\Phi_{{\bf S}_k})(I)={\bf P}_\CC,
   $$
   where ${\bf P}_\CC$ is the
 orthogonal projection from $\otimes_{i=1}^k F^2(H_{n_i})$ onto $\CC 1\subset \otimes_{i=1}^k F^2(H_{n_i})$, where $\CC 1$ is identified with $\CC 1\otimes\cdots \otimes \CC 1$.
  Taking into account that
   ${\bf K_{T}} { T}^*_{i,j}= ({\bf S}_{i,j}^*\otimes I)  {\bf K_{T}}
    $
for any $i\in \{1,\ldots, k\}$ and $j\in \{1,\ldots, n_i\}$, we deduce that
\begin{equation*}
\begin{split}
{\bf K_{T}^*} (P_{q_1}^{(1)}\otimes \cdots \otimes P_{q_k}^{(k)}\otimes I_\cH){\bf K_{T}}
&={\bf K_{T}^*}\left[ (\Phi_{{\bf S}_1}^{q_1}-  \Phi_{{\bf S}_1}^{q_1+1})\circ \cdots \circ (\Phi_{{\bf S}_k}^{q_k}-  \Phi_{{\bf S}_k}^{q_k+1})(I)\otimes I_\cH \right]{\bf K_{T}}\\
&={\bf K_{T}^*}\left[ \Phi_{{\bf S}_1}^{q_1} \circ \cdots \circ \Phi_{{\bf S}_k}^{q_k}\circ(id-\Phi_{{\bf S}_k})\circ \cdots \circ (id-\Phi_{{\bf S}_1})(I) \otimes I_\cH \right]{\bf K_{T}}\\
&=\Phi_{T_1}^{q_1}\circ \cdots \circ \Phi_{T_k}^{q_k}\circ(id-\Phi_{T_1})\circ \cdots \circ (id-\Phi_{T_k})({\bf K_{T}^*}{\bf K_{T}}).
\end{split}
\end{equation*}
On the other hand, since
\begin{equation} \label{KTKT}
{\bf K_{T}^*}{\bf K_{T}}=
\lim_{q_k\to\infty}\ldots \lim_{q_1\to\infty}  (id-\Phi_{T_k}^{q_k})\circ \cdots \circ (id-\Phi_{T_1}^{q_1})(I),
\end{equation}
where the limits are in the weak  operator topology, we can prove that
\begin{equation} \label{id-Q}
 (id-\Phi_{T_1})\circ \cdots \circ (id-\Phi_{T_k})({\bf K_{T}^*}{\bf K_{T}})={\bf \Delta_{T}}(I).
\end{equation}
Indeed,   note that    $\{(id-\Phi_{T_k}^{q_k})\circ \cdots \circ (id-\Phi_{T_1}^{q_1})(I)\}_{{\bf q}=(q_1,\ldots, q_k)\in \ZZ_+^k}$ is an  increasing sequence of positive operators and
 $$
 (id-\Phi_{T_k}^{q_k})\circ \cdots \circ (id-\Phi_{T_1}^{q_1})(I)
 =\sum_{s_k=0}^{q_k-1}\Phi_{T_k}^{s_k}\circ\cdots \sum_{s_1=0}^{q_1-1}\Phi_{T_1}^{s_1}\circ
(id-\Phi_{T_k})\circ \cdots \circ (id-\Phi_{T_1})(I).
$$
 Since $\Phi_{T_1},\ldots, \Phi_{T_k}$ are commuting WOT-continuous completely positive linear maps  and  $\lim_{q_i\to \infty}\Phi_{T_i}^{q_i}(I)$ exists  in the weak operator topology for each $i\in\{1,\ldots,k\}$, we have
\begin{equation*}
\begin{split}
(id-\Phi_{T_1})({\bf K_{T}^*}{\bf K_{T}})
&=\lim_{q_k\to\infty}\ldots \lim_{q_1\to\infty}  (id-\Phi_{T_k}^{q_k})\circ \cdots \circ (id-\Phi_{T_1}^{q_1})\circ(id-\Phi_{T_1})(I)\\
&=\lim_{q_k\to\infty}\ldots \lim_{q_2\to\infty}  (id-\Phi_{T_k}^{q_k})\circ \cdots \circ (id-\Phi_{T_2}^{q_2})
\left[\lim_{q_1\to\infty}(id-\Phi_{T_1}^{q_1})\circ(id-\Phi_{T_1})(I)\right]\\
&=\lim_{q_k\to\infty}\ldots \lim_{q_2\to\infty}  (id-\Phi_{T_k}^{q_k})\circ \cdots \circ (id-\Phi_{T_2}^{q_2})\circ(id-\Phi_{T_1})(I).
\end{split}
\end{equation*}
Applying now $id-\Phi_{T_2}$, a similar reasoning leads to
$$
(id-\Phi_{T_2})\circ(id-\Phi_{T_1})({\bf K_{T}^*}{\bf K_{T}})=\lim_{q_k\to\infty}\ldots \lim_{q_3\to\infty}  (id-\Phi_{T_k}^{q_k})\circ \cdots \circ (id-\Phi_{T_3}^{q_3})\circ(id-\Phi_{T_1})\circ(id-\Phi_{T_2})(I).
$$
Continuing this process, we obtain relation \eqref{id-Q}. Combining \eqref{id-Q} with the relation preceding \eqref{KTKT},  we have
\begin{equation*}
 {\bf K_{T}^*} (P_{q_1}^{(1)}\otimes \cdots \otimes P_{q_k}^{(k)}\otimes I){\bf K_{T}}=
  \Phi_{T_1}^{q_1}\circ \cdots \circ \Phi_{T_k}^{q_k}({\bf \Delta_{T}}(I))
  \end{equation*}
  for any $q_1,\ldots,q_k\in \ZZ^+$, which implies
  \begin{equation*}
  \begin{split}
{\bf K_{T}^*}(P_{\leq(q_1,\ldots, q_k)}  \otimes I){\bf K_{T}}&=
\sum_{s_k=0}^{q_k}\Phi_{T_k}^{s_k}\circ\cdots \sum_{s_1=0}^{q_1}\Phi_{T_1}^{s_1}\circ
(id-\Phi_{T_k})\circ \cdots \circ (id-\Phi_{T_1})(I)\\
&=(id-\Phi_{T_1}^{q_1+1})\circ \cdots \circ (id-\Phi_{T_k}^{q_k+1})(I)
\end{split}
\end{equation*}
and proves relation \eqref{KPK}. Now, one can easily complete the proof.
\end{proof}

We remark that,  in Theorem \ref{Euler1},  the limit over  ${\bf q}=(q_1,\ldots, q_k)\in \ZZ_+^k$ can be replaced by the limit over any  increasing  cofinal sequence in $\ZZ_+^k$,  or by any iterated limit  $\lim_{q_{\sigma(1)}\to\infty}\cdots\lim_{q_{\sigma(k)}\to\infty}$,
where $\sigma$ is a permutation of $\{1,\ldots, k\}$.

\begin{corollary} If ${\bf T}\in  {\bf B_n}(\cH)$ has   finite rank defect, then
$$
0\leq \text{\rm curv}({\bf T})\leq  \chi({\bf T})\leq \rank [{\bf \Delta_{T}}(I)].
$$
\end{corollary}
\begin{proof}
 According to \cite{Po-curvature-polyballs}, we have $$\text{\rm curv}({\bf T})=\lim_{q_1\to\infty}\cdots\lim_{q_k\to\infty}
\frac{\text{\rm trace\,}\left[ {\bf K_{T}^*}(P_{\leq(q_1,\ldots, q_k)}  \otimes I_\cH){\bf K_{T}}\right]}{\text{\rm trace\,}\left[P_{\leq(q_1,\ldots, q_k)}\right]}.$$
 Since  $\text{\rm trace\,}\left[ P_{q_1}^{(1)}\otimes \cdots \otimes P_{q_k}^{(k)}\right]=\rank\left[ P_{q_1}^{(1)}\otimes \cdots \otimes P_{q_k}^{(k)}\right]$ and ${\bf K_{T}^*}(P_{\leq(q_1,\ldots, q_k)}  \otimes I_\cH){\bf K_{T}}$ is a positive contraction, Theorem \ref{Euler1} implies $\text{\rm curv}({\bf T})\leq  \chi({\bf T})$. The inequality $\chi({\bf T})\leq \rank [{\bf \Delta_{T}}(I)]$ is due to the inequality \eqref{ni3}, in the particular case when $\cD={\bf \Delta_{T}}(I)^{1/2} (\cH)$.
\end{proof}

\begin{corollary} If ${\bf T}\in  {\bf B_n}(\cH)$  and ${\bf T}'\in {\bf B_n}(\cH')$  have finite rank defects,   then ${\bf T}\oplus {\bf T}'\in {\bf B_n}(\cH\oplus\cH')$ has finite rank  defect and
$$
 \chi({\bf T}\oplus {\bf T}')=\chi({\bf T})+  \chi({\bf T}').
$$
If, in addition, $\dim\cH'<\infty$, then $ \chi({\bf T}\oplus {\bf T}')=\chi({\bf T})$.
\end{corollary}

Given a function $\kappa :\NN\to \NN$ and ${\bf n}^{(i)}:=(n_1^{(i)},\ldots, n_{\kappa(i)}^{(i)})\in \NN^{\kappa(i)}$ for  $i\in \{1,\ldots,p\}$,
we consider the polyball ${\bf B}_{{\bf n}^{(i)}}(\cH_i)$, where $\cH_i$ is a Hilbert space.
 Let ${\bf X}^{(i)}\in{\bf B}_{{\bf n}^{(i)}}(\cH_i)$ with ${\bf X}^{(i)}:=(X_1^{(i)},\ldots, X_{\kappa(i)}^{(i)})$ and
 $X_r^{(i)}:=(X_{r,1}^{(i)},\ldots, X^{(i)}_{r,n_r^{(i)}})\in B(\cH_i)^{n_r^{(i)}}$  for
 $r\in \{1,\ldots, \kappa(i)\}$.
 If ${\bf X}:=({\bf X}^{(1)},\ldots, {\bf X}^{(p)})\in {\bf B}_{{\bf n}^{(1)}}(\cH_1)\times \cdots \times {\bf B}_{{\bf n}^{(p)}}(\cH_p)$, we define the ampliation
 $\widetilde {\bf X}$ by setting
 $\widetilde {\bf X}:=(\widetilde {\bf X}^{(1)},\ldots, \widetilde {\bf X}^{(p)})$, where
 $\widetilde {\bf X}^{(i)}:=(\widetilde X_1^{(i)},\ldots, \widetilde X_{\kappa(i)}^{(i)})$ and
 $\widetilde X_r^{(i)}:=(\widetilde X_{r,1}^{(i)},\ldots, \widetilde X^{(i)}_{r,n_r^{(i)}})$  for
 $r\in \{1,\ldots, \kappa(i)\}$,  and
 $$
 \widetilde X_{r,s}^{(i)}:=I_{\cH_1}\otimes \cdots \otimes I_{\cH_{i-1}}\otimes X_{r,s}^{(i)}\otimes I_{\cH_{i+1}}\otimes I_{\cH_p}
 $$
 for all $i\in \{1,\ldots,p\}$, $r\in \{1,\ldots, \kappa(i)\}$, and $s\in \{1,\ldots, n_r^{(i)}\}$.

 \begin{theorem} \label{multiplicative-prop} Let  ${\bf X}:=({\bf X}^{(1)},\ldots, {\bf X}^{(p)})\in {\bf B}_{{\bf n}^{(1)}}(\cH_1)\times \cdots \times {\bf B}_{{\bf n}^{(p)}}(\cH_p)$ be such that each ${\bf X}^{(i)}$ has finite rank defect. Then the ampliation $\widetilde {\bf X}$ is in the regular polyball ${\bf B}_{({\bf n}^{(1)},\ldots, {\bf n}^{(p)})}(\cH_1\otimes \cdots \otimes \cH_p)$, has finite rank defect,
 and the Euler characteristic satisfies the relation
 $$
 \chi(\widetilde {\bf X})=\prod_{i=1}^p \chi({\bf X}^{(i)}).
 $$
 \end{theorem}
 \begin{proof} For each  ${\bf m}^{(i)}\in \ZZ_+^{\kappa(i)}$ with $0\leq {\bf m}^{(i)}\leq (1,\ldots, 1)$ and $i\in \{1,\ldots, p\}$, we have
 $$
 {\bf \Delta}_{\widetilde {\bf X}}^{({\bf m}^{(1)},\ldots, {\bf m}^{(p)})}(I_{\cH_1\otimes \cdots \otimes \cH_p})
 ={\bf \Delta}^{{\bf m}^{(1)}}_{{\bf X}^{(1)}}(I_{\cH_1})\otimes \cdots \otimes {\bf \Delta}^{{\bf m}^{(p)}}_{{\bf X}^{(p)}}(I_{\cH_p})\geq 0.
 $$
  Consequently, $\widetilde {\bf X}$ is in the regular polyball ${\bf B}_{({\bf n}^{(1)},\ldots, {\bf n}^{(p)})}(\cH_1\otimes \cdots \otimes \cH_p)$ and
 $$
 \rank\left[{\bf \Delta}_{\widetilde {\bf X}} (I_{\cH_1\otimes \cdots \otimes \cH_p})\right]
 =\rank\left[{\bf \Delta}_{{\bf X}^{(1)}}(I_{\cH_1})\right]\cdots \rank\left[{\bf \Delta}_{{\bf X}^{(p)}}(I_{\cH_p})\right]<\infty.
 $$
 Let ${\bf q}^{(i)}:=(q_1^{(i)}, \ldots, q_{\kappa(i)}^{(i)})\in \ZZ_+^{\kappa(i)} $
 for $i\in \{1,\ldots, p\}$. According to Theorem \ref{Euler1}, the Euler characteristic $\chi(\widetilde{\bf X})$ is equal
\begin{equation*}
\begin{split}
\lim
\frac{\rank\left[(id- \Phi_{X_1^{(1)}}^{q_1^{(1)}+1})\circ \cdots \circ
(id-\Phi_{X_{\kappa(1)}^{(1)}}^{q_{\kappa(1)}^{(1)}+1})(I_{\cH_1})
\otimes \cdots \otimes
(id-\Phi_{X_1^{(p)}}^{q_1^{(p)}+1})\circ \cdots \circ
(id- \Phi_{X_{\kappa(p)}^{(p)}}^{q_{\kappa(p)}^{(p)}+1})(I_{\cH_p})\right]}
{\prod_{r=1}^{\kappa(1)}(1+n_r^{(1)}+\cdots +(n_r^{(1)})^{q^{(1)}_r})\cdots \prod_{r=1}^{\kappa(p)}(1+n_r^{(p)}+\cdots +(n_r^{(p)})^{q^{(p)}_r})},
\end{split}
\end{equation*}
where the limit is taken  over ${{\bf q}^{(1)}\in \ZZ_+^{\kappa(1)},\ldots,{\bf q}^{(p)}\in \ZZ_+^{\kappa(p)}}$,
which is equal to to the product
$$
\prod_{i=1}^p\lim_{{\bf q}^{(i)}\in \ZZ_+^{\kappa(i)}}
\frac{\rank\left[ (id- \Phi_{X_1^{(i)}}^{q_1^{(i)}+1})\circ \cdots \circ
(id-\Phi_{X_{\kappa(i)}^{(i)}}^{q_{\kappa(i)}^{(i)}+1})(I_{\cH_i})
           \right]}
{\prod_{r=1}^{\kappa(i)}(1+n_r^{(i)}+\cdots +(n_r^{(i)})^{q^{(i)}_r})}.
$$
Due to Theorem \ref{Euler1}, the latter  product is equal to
$\prod_{i=1}^p \chi({\bf X}^{(i)})$.
The proof is complete.
\end{proof}

\bigskip

\section{Invariant subspaces, Euler characteristic, and classification}

In this section, we give a  characterization of the invariant subspaces of the tensor product $\otimes_{i=1}^k F^2(H_{n_i})$ with positive defect operators. We show that the Euler characteristic completely classifies the  finite rank  Beurling type invariant subspaces  of $\otimes_{i=1}^k F^2(H_{n_i})$ and prove some of its basic properties, including the fact that its  range   is the interval $[0,\infty)$.

Let ${\bf S}:=({\bf S}_1,\ldots, {\bf S}_k)$, where ${\bf S}_i:=({\bf S}_{i,1},\ldots, {\bf S}_{i, n_i})$, be the universal model of the abstract polyball ${\bf B_n}$, and let $\cH$ be a Hilbert space. We say that $\cM$
 is an invariant subspace of $\otimes_{i=1}^k F^2(H_{n_i})\otimes
  \cH$ or  that $\cM$  is    invariant  under  ${\bf S}\otimes I_\cH$ if it is invariant under each operator
  ${\bf S}_{i,j}\otimes I_\cH$ for $i\in \{1,\ldots, k\}$ and $j\in \{1,\ldots, n_j\}$.

 \begin{definition} Given two invariant subspaces $\cM$ and $\cN$ under ${\bf S}\otimes I_\cH$,  we say that they are unitarily equivalent if there is a unitary operator $U:\cM\to \cN$ such that $U({\bf S}_{i,j}\otimes I_\cH)|_\cM=({\bf S}_{i,j}\otimes I_\cH)|_\cN U$ for any $i\in \{1,\ldots, k\}$  and  $j\in \{1,\ldots, n_i\}$.
 \end{definition}

We   define
 the {\it right creation  operator} $R_{i,j}:F^2(H_{n_i})\to
F^2(H_{n_i})$     by setting
\begin{equation*}
R_{i,j} e_\alpha^i:=  e^i_ {\alpha {g_j}}, \qquad \alpha\in \FF_{n_i}^+,
\end{equation*}
 and
 the operator ${\bf R}_{i,j}$ acting on the tensor Hilbert space
$F^2(H_{n_1})\otimes\cdots\otimes F^2(H_{n_k})$. We also denote ${\bf R}:=({\bf R}_1,\ldots, {\bf R}_k)$, where  ${\bf R}_i:=({\bf R}_{i,1},\ldots,{\bf R}_{i,n_i})$.
Let $\varphi=\sum_{\beta_i\in \FF_{n_i}^+} c_{\beta_1,\ldots,\beta_k}
   e^1_{\beta_1}\otimes \cdots \otimes  e^k_{\beta_k}$ be in $\otimes_{i=1}^kF^2(H_{n_i})$ and consider the formal power series
   $\widetilde\varphi({\bf R}_{i,j}):=\sum_{\beta_i\in \FF_{n_i}^+} c_{\beta_1,\ldots,\beta_k}
   {\bf R}_{1,\widetilde\beta_1} \cdots  {\bf R}_{k,\widetilde
   \beta_k}$, where
$\widetilde \beta_i=g_{j_p}^i\cdots g_{j_1}^i$ denotes the reverse of
$\beta_i=g_{j_1}^i\cdots g_{j_p}^i\in \FF_{n_i}^+$.
   Note that  $\widetilde\varphi({\bf R}_{i,j})(e^1_{\gamma_1}\otimes \cdots \otimes  e^k_{\gamma_k})$ is in $\otimes_{i=1}^kF^2(H_{n_i})$. We say that $\varphi$ is a right multiplier of $\otimes_{i=1}^kF^2(H_{n_i})$ if
   $$
   \sup_{p\in \cP, \|p\|\leq 1}\| \widetilde\varphi({\bf R}_{i,j})p\|<\infty,
   $$
   where $\cP$ is the set of all polynomials $\sum a_{\alpha_1,\ldots,\alpha_k}
   e^1_{\alpha_1}\otimes \cdots \otimes  e^k_{\alpha_k}$ in $\otimes_{i=1}^kF^2(H_{n_i})$. In this case, there is a unique bounded operator acting on $\otimes_{i=1}^kF^2(H_{n_i})$, denoted also by  $\widetilde\varphi({\bf R}_{i,j})$, such that
  $$
   \widetilde\varphi({\bf R}_{i,j})p=\sum_{\beta_i\in \FF_{n_i}^+} c_{\beta_1,\ldots,\beta_k}
   {\bf R}_{1,\widetilde\beta_1} \cdots  {\bf R}_{k,\widetilde
   \beta_k}p,\qquad p\in \cP.
   $$
The set of all  operators $\widetilde\varphi({\bf R}_{i,j})$ satisfying the conditions above is a Banach algebra denoted by $R^\infty({\bf B_n})$. We proved in \cite{Po-Berezin-poly} that $F^\infty({\bf B_n})'=R^\infty({\bf B_n})$ and
$F^\infty({\bf B_n})''=F^\infty({\bf B_n})$, where $'$ stands for the commutant.

According to Theorem 5.1 from \cite{Po-Berezin-poly}, if a subspace $\cM\subset \otimes_{i=1}^k F^2(H_{n_i})\otimes \cH$ is  co-invariant  under each operator ${\bf S}_{i,j}\otimes I_\cH$, then
\begin{equation*}
\overline{\text{\rm span}}\,\left\{\left({\bf S}_{1,\beta_1}\cdots {\bf S}_{k,\beta_k}\otimes
I_\cH\right)\cM:\ \beta_1\in \FF_{n_1}^+,\ldots, \beta_k\in \FF_{n_k}^+\right\}=\otimes_{i=1}^k F^2(H_{n_i})\otimes \cE,
\end{equation*}
where $\cE:=({\bf P}_\CC\otimes I_\cH)(\cM)$.
Consequently,
 a subspace
$\cR\subseteq \otimes_{i=1}^k F^2(H_{n_i})\otimes \cH$ is
 reducing under
${\bf S}\otimes I_\cH$
if and only if   there exists a subspace $\cG\subseteq \cH$ such
that
$
 \cR=\otimes_{i=1}^k F^2(H_{n_i})\otimes \cG.
$

 It is well known \cite{Ru} that the lattice of the invariant subspaces for the Hardy space $H^2(\DD^k)$ is very complicated and contains many  invariant subspaces which are not of Beurling type. The same complicated situation occurs in the case of the tensor product $\otimes_{i=1}^k F^2(H_{n_i})$.
Following the classical case \cite{Be}, we say that $\cM$ is a {\it Beurling type invariant subspace} for ${\bf S}\otimes I_\cH$ if there is  an inner  multi-analytic operator
 $\Psi:\otimes_{i=1}^k F^2(H_{n_i})\otimes \cE \to
 \otimes_{i=1}^k F^2(H_{n_i})\otimes \cH$ with respect to ${\bf S}$,
 where  $\cE$ is a Hilbert space, such that
$\cM=\Psi\left[\otimes_{i=1}^k F^2(H_{n_i})\otimes \cE\right]$. In this case, $\Psi$ can be chosen to be isometric. We recall that $\Psi$ is  multi-analytic  if $\Psi({\bf S}_{i,j}\otimes I_\cH)=({\bf S}_{i,j}\otimes I_\cE)\Psi$ for all $i,j$. More about multi-analytic operators on Fock spaces can be found in \cite{Po-charact} and \cite{Po-analytic}.
In \cite{Po-Berezin-poly}, we proved that $\cM$ is a Beurling type invariant subspace  for ${\bf S}\otimes I_\cH$
if and only if   $$
   (id-\Phi_{{\bf S}_1\otimes I_\cH})\circ \cdots \circ
    (id-\Phi_{{\bf S}_k\otimes I_\cH})(P_\cM)\geq 0,
   $$
   where $P_\cM$ is the orthogonal projection onto $\cM$.
We introduce the {\it defect operator}  of $\cM$ be setting $\Delta_\cM:={\bf \Delta_{S\otimes {\it I}_\cH}}(P_\cM)$.

In what follows, we present a more direct and more transparent proof for the characterization of the  invariant subspaces of $\otimes_{i=1}^kF^2(H_{n_i})$ with positive defect operator,  which complements the corresponding result from \cite{Po-Berezin-poly} (see Corollary 5.3).

\begin{theorem}\label{Beurling} A subspace  $\cM\subseteq \otimes_{i=1}^kF^2(H_{n_i})$, $n_i\geq 1$,  is  invariant under the universal model ${\bf S}$ and    has positive  defect operator $\Delta_\cM$ if and only if there is a sequence $\{\psi_s\}_{s=1}^N$, where $N\in \NN$ or $N=\infty$,  of right multipliers of  $\otimes_{i=1}^kF^2(H_{n_i})$ such that $\{\widetilde\psi_s({ \bf R}_{i,j})\}_{s=1}^N$ are isometries with orthogonal ranges and
$$
P_\cM=\sum_{s=1}^N \widetilde\psi_s({\bf R}_{i,j})\widetilde\psi_s({\bf R}_{i,j})^*,
$$
 where the convergence is in the strong operator topology. The sequence $\{\psi_s\}_{s=1}^N$ is uniquely determined up to a unitary equivalence.

In addition, we can choose the sequence $\{\psi_s\}_{s=1}^N$ such that  each $\psi_s$ is in the range of $\Delta_\cM$.
We  also have $N=\rank(\Delta_\cM^{1/2})$
and
$$
  \Delta_\cM\xi=\sum_{t=1}^N\left<\xi, \psi_s\right>\psi_s, \qquad \xi\in \otimes_{i=1}^kF^2(H_{n_i}).
  $$
\end{theorem}
\begin{proof} Assume that $\cM\subseteq \otimes_{i=1}^kF^2(H_{n_i})$ is an invariant subspace under ${\bf S}$ and $\Delta_\cM:={\bf \Delta_{S}}(P_\cM)\geq 0$.
Let  ${\bf A}=({ A}_1,\ldots, { A}_k)$ with $A_i:=({\bf S}_{i,1}|_\cM,\ldots, {\bf S}_{i,n_i}|_\cM)$ and note that ${\bf A}$ is a pure element in ${\bf B_n}(\cM)$.
 The  noncommutative Berezin kernel associated with
   ${\bf A}$,
   $${\bf K_{A}}: \cM \to \left(\otimes_{i=1}^kF^2(H_{n_i})\right) \otimes  \overline{{\bf \Delta_{A}}(I_\cM)^{1/2} (\cM)},
   $$
is an isometry. Note that $\Delta_\cM(\cM)\subseteq \cM$ and  $\Delta_\cM|_\cM={\bf \Delta_{A}}(I_\cM)$.  Consequently,  we have $\Delta_\cM(\cM^\perp)\subseteq \cM^\perp$ and
 ${\bf \Delta_{A}}(I_\cM)^{1/2} (\cM)=\Delta_\cM^{1/2}(\otimes_{i=1}^kF^2(H_{n_i}))$.
Consider the extended Berezin kernel
$${\bf \widehat K_{A}}: \otimes_{i=1}^kF^2(H_{n_i}) \to \left(\otimes_{i=1}^kF^2(H_{n_i})\right) \otimes  \overline{\Delta_\cM^{1/2}(\otimes_{i=1}^kF^2(H_{n_i}))}
   $$
   defined by ${\bf \widehat K_{A}}|_\cM:={\bf K_{A}}$ and ${\bf \widehat K_{A}}|_{\cM^\perp}:=0$.
  Let $\{\chi_s\}_{s=1}^N$ be an orthonormal basis for the defect space $\overline{\Delta_\cM^{1/2}(\otimes_{i=1}^kF^2(H_{n_i}))}$ and  set
  $$
  \psi_s:={\bf \widehat K_{A}^*}(1\otimes \chi_s)={\bf K_{A}^*}(1\otimes \chi_s)=
  {\bf \Delta_{A}}(I_\cM)^{1/2}\chi_s=
  \Delta_\cM^{1/2}\chi_s.
  $$
 If $p=\sum a_{\alpha_1,\ldots, \alpha_k} e^1_{\alpha_1}\otimes\cdots \otimes  e^k_{\alpha_k} $ is any polynomial in  $\otimes_{i=1}^kF^2(H_{n_i})$, we have
 \begin{equation*}
 \begin{split}
 {\bf K_{A}^*}(p\otimes \chi_s)=p({\bf S}_{i,j}|_\cM) \Delta_\cM^{1/2}\chi_s=
 p({\bf S}_{i,j})\psi_s=p\psi_s=\widetilde\psi_s({\bf R}_{i,j})p.
 \end{split}
 \end{equation*}
 Hence,
 \begin{equation*}
 \|\widetilde\psi_s({\bf R}_{i,j})p\|=\|{\bf K_{A}^*}(p\otimes \chi_s)\|\leq \|p\|.
 \end{equation*}
  Since the polynomials are dense in $\otimes_{i=1}^kF^2(H_{n_i})$, we deduce that $ \psi_s$ is a  right multiplier of  $\otimes_{i=1}^kF^2(H_{n_i})$.
  For each $s\in \{1,\ldots, N\}$,  we define the linear operator
  $\Gamma_s: \otimes_{i=1}^kF^2(H_{n_i})\to \otimes_{i=1}^kF^2(H_{n_i})$ by setting
  $$
  \left< \Gamma_sf,g\right>:=\left< {\bf \widehat K_{A}} f, g\otimes \chi_s\right>, \qquad f,g\in \otimes_{i=1}^kF^2(H_{n_i}).
  $$
   Note that
   \begin{equation*}
   \begin{split}
   \sum_{s=1}^N \|\Gamma_sf\|^2&=\sum_{s=1}^N \sum_{\alpha_i\in \FF_{n_i}^+}
   |\left< \Gamma_sf,e^1_{\alpha_1}\otimes \cdots \otimes e^k_{\alpha_k}\right>|^2
   =\sum_{s=1}^N \sum_{\alpha_i\in \FF_{n_i}^+}
   \left|\left<{\bf \widehat K_{A}}f,(e^1_{\alpha_1}\otimes \cdots \otimes e^k_{\alpha_k})\otimes \chi_s\right>\right|^2\\
   &=\|{\bf \widehat K_{A}}f\|^2\leq \|{\bf K_A} f\|^2=\|f\|^2
   \end{split}
   \end{equation*}
   for any $f\in \otimes_{i=1}^kF^2(H_{n_i})$. Consequently,
   $
   \sum_{s=1}^N \Gamma_s^* \Gamma_s
   $
   is convergent in the strong operator topology and
   $$
    \sum_{s=1}^N \Gamma_s^* \Gamma_s={\bf \widehat K_{A}^*}{\bf \widehat K_{A}}=
    \left[\begin{matrix} {\bf K_{A}^*}{\bf K_{A}}&0\\
    0&0 \end{matrix}\right]=\left[\begin{matrix} I_\cM&0\\
    0&0 \end{matrix}\right]=P_\cM.
   $$
    Note also that
   \begin{equation*}
   \begin{split}
   \left< \Gamma_s f,p\right>=\left<f, {\bf K_A^*}(p\otimes \chi_s)\right>
   =\left<f,  \widetilde\psi_s({\bf R}_{i,j}) p\right>=\left<\widetilde\psi_s({\bf R}_{i,j})^*f,   p\right>.
   \end{split}
   \end{equation*}
  Since the polynomials are dense in $\otimes_{i=1}^kF^2(H_{n_i})$, we deduce that
  $\Gamma_s=\widetilde\psi_s({\bf R}_{i,j})^*$ and, consequently, $\sum_{s=1}^N \widetilde\psi_s({\bf R}_{i,j})\widetilde\psi_s({\bf R}_{i,j})^* =P_\cM$, where the  convergence is  in the strong operator topology.
  On the other hand, since $\Phi_{{\bf S}_i}$ is WOT-continuous, so is the defect map
  ${\bf \Delta_S}:=(id-\Phi_{{\bf S}_1})\circ\cdots\circ (id -\Phi_{{\bf S}_k})$. Consequently,
  \begin{equation*}
  \begin{split}
  \Delta_\cM&={\bf \Delta_{S}}(P_\cM)=\sum_{s=1}^N \widetilde\psi_s({\bf R}_{i,j}){\bf \Delta_S}(I)\widetilde\psi_s({\bf R}_{i,j})^* \\
  &=\sum_{s=1}^N \widetilde\psi_s({\bf R}_{i,j})P_\CC\widetilde\psi_s({\bf R}_{i,j})^*.
  \end{split}
  \end{equation*}
 Hence, we deduce that
 $$
  \Delta_\cM\xi=\sum_{s=1}^N \widetilde\psi_s({\bf R}_{i,j})\left<\widetilde\psi_s({\bf R}_{i,j})^*\xi, 1\right>
  =\sum_{t=1}^N\left<\xi, \psi_s\right>\psi_s, \qquad \xi\in \otimes_{i=1}^kF^2(H_{n_i}).
  $$
Let
    $\Psi:(\otimes_{i=1}^k F^2(H_{n_i}))\otimes
\CC^N \to \otimes_{i=1}^k F^2(H_{n_i})$  be the bounded operator having the $1\times N$ matrix representation
$$
[\widetilde\psi_1({\bf R}_{i,j}), \widetilde\psi_2({\bf R}_{i,j}), \ldots],
$$
  where $\CC^\infty$ stands for $\ell^2(\NN)$.
 Note that $\Psi$ is a multi-analytic operator with respect to the universal model ${\bf S}$ and $\Psi \Psi^*=P_\cM$. Therefore, $\Psi$ is a partial isometry.
Since ${\bf S}_{i,j}$ are isometries, the initial space of $\Psi$, i.e.
$$\Psi^*(\otimes_{i=1}^k F^2(H_{n_i}))
 =\{x\in (\otimes_{i=1}^k F^2(H_{n_i}))\otimes \CC^N: \ \|\Psi x\|=\|x\|\}$$
  is
 reducing under each operator ${\bf S}_{i,j}\otimes I_{\CC^N}$.
 Consequently, since ${\bf \Delta_{S}}(I)=P_\CC$,  to prove that $\Psi$ is an isometry, it is enough to show that
$$\CC= P_\CC \widetilde\psi_s({\bf R}_{i,j})^* (\otimes_{i=1}^k F^2(H_{n_i}))
 \quad \text{ for each }  s\in \{1,\ldots, N\}.
 $$
The latter equality is true since
$$
P_\CC\widetilde\psi_s({\bf R}_{i,j})^* ( \psi_s)=\left<\widetilde\psi_s({\bf R}_{i,j})^* (\psi_s),1\right>=\| \psi_s\|^2=1.
$$
Therefore, $\Psi$ is an isometry.

  To prove the converse of the theorem, assume that there is a sequence $\{\psi_s\}_{s=1}^N$ of right multipliers of  $\otimes_{i=1}^kF^2(H_{n_i})$ such that $\{\widetilde\psi_s({ \bf R}_{i,j})\}_{s=1}^N$ are isometries with orthogonal ranges and
$$
P_\cM=\sum_{s=1}^N \widetilde\psi_s({\bf R}_{i,j})\widetilde\psi_s({\bf R}_{i,j})^*.
$$
Since ${\bf S}_{i,j}$ commutes with ${\bf R}_{r,q}$ for any $i,r\in \{1,\ldots, k\}$, $j\in \{1,\ldots, n_i\}$, and $q\in \{1,\ldots, n_r\}$,  and $\{\widetilde\psi_s({\bf R}_{i,j})\}_{s=1}^N$ are isometries with orthogonal ranges,  we deduce that
$P_\cM {\bf S}_{i,j}P_\cM= {\bf S}_{i,j}P_\cM$. Therefore, $\cM$ is an  invariant subspace under the universal model ${\bf S}$. Moreover, we have
$$
\Delta_\cM=(id-\Phi_{{\bf S}_1})\circ\cdots\circ (id -\Phi_{{\bf S}_k})(P_\cM)=\sum_{s=1}^N \widetilde\psi_s({\bf R}_{i,j}){\bf \Delta_S}(I)\widetilde\psi_s({\bf R}_{i,j})^*\geq 0.
$$

Now, we prove that the sequence $\{\psi_s\}_{s=1}^N$ is uniquely determined up to a unitary equivalence.   Let $\{\psi_s'\}_{s=1}^{N'}$  be another sequence
of right multipliers of  $\otimes_{i=1}^kF^2(H_{n_i})$ such that $\{\widetilde\psi_s'({\bf R}_{i,j})\}_{s=1}^{N'}$ are isometries with orthogonal ranges and
$$
P_\cM=\sum_{s=1}^{N'} \widetilde\psi_s'({\bf R}_{i,j})\widetilde\psi_s'({\bf R}_{i,j})^*.
$$
As above, we have
$$
  \Delta_\cM\xi=
  \sum_{t=1}^N\left<\xi, \psi_s\right>\psi_s=\sum_{t=1}^{N'}\left<\xi, \psi_s'\right>\psi_s', \qquad \xi\in \otimes_{i=1}^kF^2(H_{n_i}).
  $$
 Consequently, if $\Delta_\cM$ has finite rank then $\{\psi_s\}_{s=1}^N$ and $\{\psi_s'\}_{s=1}^{N'}$    are orthonormal bases  for the range of
 $\Delta_\cM$, thus $N=N'\in \NN$.    If $\Delta_\cM$  does not have finite rank, similar arguments, show that $N=N'=\infty$.
Now, note that
$$
\cM=\bigoplus_{s=1}^N \left(\otimes_{i=1}^kF^2(H_{n_i})\right)\psi_s=\bigoplus_{s=1}^N \left(\otimes_{i=1}^kF^2(H_{n_i})\right)\psi_s',
$$
and $\{(e^1_{\alpha_1}\otimes \cdots \otimes e^k_{\alpha_k})\psi_s:\ \alpha_i\in \FF_{n_i}^+, s\in \{1,\ldots, N\}\}$ and $\{(e^1_{\alpha_1}\otimes \cdots \otimes e^k_{\alpha_k})\psi_s':\ \alpha_i\in \FF_{n_i}^+, s\in \{1,\ldots, N\}\}$
 are orthonormal bases for $\cM$.
 Define the unitary operator $U:\cM\to \cM$ by setting
 $$U(  (e^1_{\alpha_1}\otimes \cdots \otimes e^k_{\alpha_k})\psi_s):=(e^1_{\alpha_1}\otimes \cdots \otimes e^k_{\alpha_k})\psi_s.
 $$
 It is clear than $U({\bf S}_{i,j}|_\cM)=({\bf S}_{i,j}|_\cM)U$ for any
 $i\in \{1,\ldots, k\}$ and  $j\in \{1,\ldots, n_i\}$. The proof is complete.
 \end{proof}

   If $\cM$  is a  Beurling type invariant subspace  of  ${\bf S}\otimes I_\cH$, then $({\bf S}\otimes I_\cH)|_\cM:=(({\bf S}_1\otimes I_\cH)|_\cM,\ldots, ({\bf S}_k\otimes I_\cH)|_\cM)$ is in the polyball ${\bf B_n}(\cM)$, where  $({\bf S}_i\otimes I_\cH)|_\cM:=(({\bf S}_{i,1}\otimes I_\cH)|_\cM,\ldots, ({\bf S}_{i, n_i}\otimes I_\cH)|_\cM)$. We say that $\cM$ has finite rank if $({\bf S}\otimes I_\cH)|_\cM$ has finite rank.
The next result shows that the Euler characteristic  completely classifies the
 finite rank  Beurling type invariant subspaces  of ${\bf S}\otimes I_\cH$  which do not contain reducing subspaces. In particular, the Euler characteristic classifies  the  finite rank  Beurling type invariant subspaces  of $ F^2(H_{n_1})\otimes\cdots \otimes F^2(H_{n_k})$.

\begin{theorem}\label{classification} Let $\cM$ and $\cN$ be  invariant subspaces  of $\otimes_{i=1}^k F^2(H_{n_i})\otimes \cH$ which do not contain reducing subspaces of ${\bf S}\otimes I_\cH$ and such that the defect operators $\Delta_\cM$ and $\Delta_\cN$ are positive and  have finite ranks. Then $\cM$ and $\cN$ are unitarily equivalent if and only if
$$
\chi(({\bf S}\otimes I_\cH)|_\cM)=\chi(({\bf S}\otimes I_\cH)|_\cN).
$$
\end{theorem}
\begin{proof}
 Let $\cM$ be a   Beurling type invariant subspaces  of $\otimes_{i=1}^k F^2(H_{n_i})\otimes \cH$. Then there is a Hilbert space $\cL$ and an isometric multi-analytic operator $\Psi:\otimes_{i=1}^k F^2(H_{n_i})\otimes \cL\to \otimes_{i=1}^k F^2(H_{n_i})\otimes \cH$ such that
$\cM=\Psi[\otimes_{i=1}^k F^2(H_{n_i})\otimes \cL]$. Consequently, we have $P_\cM=\Psi \Psi^*$ and
\begin{equation*}
\begin{split}
{\bf \Delta}_{({\bf S}\otimes I)|_\cM}(I_\cM)&= \left(id -\Phi_{({\bf S}_1\otimes I)|_\cM}\right)\circ \cdots \circ\left(id -\Phi_{ ({\bf S}_k\otimes I)|_\cM}\right)(P_\cM)\\
&=\Psi \left(id -\Phi_{{\bf S}_1\otimes I}\right)\circ \cdots \circ\left(id -\Phi_{ {\bf S}_k\otimes I}\right)(I)\Psi^*|_\cM=\Psi({\bf P}_\CC\otimes I_\cL)\Psi^*|_\cM.
\end{split}
\end{equation*}
 If $\{\ell_\omega\}_{\omega\in \Omega}$ is  an orthonormal basis for $\cL$, then $\{v_\omega:=\Psi(1\otimes \ell_\omega):\ \omega\in \Omega\}$ is an orthonormal  set
and
$$
\{
\Psi( e^1_{\beta_1}\otimes \cdots \otimes  e^k_{\beta_k}\otimes \ell_\omega): \ \beta_i\in\FF_{n_i}^+, i\in \{1,\ldots, k\},  \omega\in \Omega\}
$$
is an orthonormal basis for $\cM$.  It is easy to see that
$\overline{\Psi({\bf P}_\CC\otimes I_\cL)\Psi^*(\cM)}$ coincides with the  closure of the range of the defect operator
${\bf \Delta}_{({\bf S}\otimes I)|_\cM}(I_\cM)$ and also to  the closed linear span of $\{v_\omega:=\Psi(1\otimes \ell_\omega):\ \omega\in \Omega\}$.
This shows that
$$
\rank [({\bf S}\otimes I)|_\cM]=\text{\rm card\,} \Omega=\dim \cL.
$$

Assume that $\rank ({\bf S}\otimes I)|_\cM)=p=\dim\cL$. Taking into account that $\Psi$ is a multi-analytic operator and $P_\cM=\Psi\Psi^*$, we deduce that
\begin{equation*}
\begin{split}
\left(id -\Phi_{({\bf S}_1\otimes I)|_\cM}^{q_1+1}\right)\circ \cdots \circ\left(id -\Phi_{ ({\bf S}_k\otimes I)|_\cM}^{q_k+1}\right)(I_\cM)
&=\left(id -\Phi_{{\bf S}_1\otimes I}^{q_1+1}\right)\circ \cdots \circ\left(id -\Phi_{ {\bf S}_k\otimes I}^{q_k+1}\right)(P_\cM)\\
&=\Psi \left(id -\Phi_{{\bf S}_1\otimes I_\cL}^{q_1+1}\right)\circ \cdots \circ\left(id -\Phi_{ {\bf S}_k\otimes I_\cL}^{q_k+1}\right)(I)\Psi^*|_\cM.
\end{split}
\end{equation*}
Hence, we have
\begin{equation*}
\begin{split}
\rank &\left[\left(id -\Phi_{({\bf S}_1\otimes I)|_\cM}^{q_1+1}\right)\circ \cdots \circ\left(id -\Phi_{ ({\bf S}_k\otimes I)|_\cM}^{q_k+1}\right)(I_\cM)\right]\\
&\qquad = \rank \left[\left(id -\Phi_{{\bf S}_1\otimes I_\cL}^{q_1+1}\right)\circ \cdots \circ\left(id -\Phi_{ {\bf S}_k\otimes I_\cL}^{q_k+1}\right)(I)\right]\\
&\qquad
=\rank\left[\left(id -\Phi_{{\bf S}_1}^{q_1+1}\right)\circ \cdots \circ\left(id -\Phi_{{\bf S}_k}^{q_k+1}\right)(I)\right]\dim\cL.
\end{split}
\end{equation*}
Due to  Theorem \ref{Euler1} and the fact that
$\chi({\bf S})=1$, we deduce that
$$\chi (({\bf S}\otimes I)|_\cM)=\dim\cL=p=\rank [({\bf S}\otimes I)|_\cM].
$$
If  $\cM$ and $\cN$ are  finite rank  Beurling type invariant subspaces  of $\otimes_{i=1}^k F^2(H_{n_i})\otimes \cH$ and $({\bf S}\otimes I)|_\cM$ is unitarily equivalent to $({\bf S}\otimes I)|_\cN$, then
$
\chi(({\bf S}\otimes I)|_\cM)=\chi(({\bf S}\otimes I)|_\cN).
$
To prove the converse, assume that the latter equality holds. Due to the first part of the proof, we must have
$\rank(({\bf S}\otimes I)|_\cM)=\rank(({\bf S}\otimes I)|_\cN)$. This shows that the defect spaces associated with $({\bf S}\otimes I)|_\cM$ and $({\bf S}\otimes I)|_\cN$ have the same dimension.  Now, the Wold decomposition from \cite{Po-Berezin-poly} implies that $({\bf S}\otimes I)|_\cM$ is unitarily equivalent to $({\bf S}\otimes I)|_\cN$.
The proof is complete.
\end{proof}

Let $\cM$ be an invariant subspace of the tensor product $F^2(H_{n_1})\otimes\cdots\otimes F^2(H_{n_k})\otimes \cE$, where $\cE$ is a finite dimensional Hilbert space. We introduce  the Euler characteristic of  of $\cM^\perp$ by setting
\begin{equation*}
\begin{split}
\chi(\cM^\perp)&:=\lim_{q_1\to\infty}\cdots\lim_{q_k\to\infty}
\frac{\rank\left[P_{\cM^\perp}(P_{\leq(q_1,\ldots, q_k)}\otimes I_\cE)  \right]}{\rank\left[P_{\leq(q_1,\ldots, q_k)}\right]}.
\end{split}
\end{equation*}

\begin{theorem}\label{Euler-cha} Let $\cM$ be an invariant subspace of the tensor product $F^2(H_{n_1})\otimes\cdots\otimes F^2(H_{n_k})\otimes \cE$, where $\cE$ is a finite dimensional Hilbert space. Then the Euler characteristic of  $\cM^\perp$ exists and satisfies the equation
 \begin{equation*}
\begin{split}
\chi(\cM^\perp)=\chi({\bf M}),
\end{split}
\end{equation*}
where  ${\bf M}:=(M_1,\ldots, M_k)$ with $M_i:=(M_{i,1},\ldots, M_{i,n_i})$ and $M_{i,j}:=P_{\cM^\perp}({\bf S}_{i,j}\otimes I_\cE)|_{\cM^\perp}$.
\end{theorem}
\begin{proof}
  Taking into account that $\cM$ is an invariant subspace under each operator  ${\bf S}_{i,j}\otimes I_\cE$ for $i\in \{1,\ldots, k\}$, $j\in\{1,\ldots, n_i\}$, we deduce that   $M_{i,j}^* M_{r,s}^*=({\bf S}_{i,j}^*\otimes I_\cE) ({\bf S}_{r,s}^*\otimes I_\cE)|_{\cM^\perp}$. Consequently, we have
\begin{equation*}
{\bf \Delta_{M}^p}(I_{\cM^\perp})=P_{\cM^\perp}{\bf \Delta_{S\otimes {\it I}}^p}(I)|_{\cM^\perp}\geq 0
\end{equation*}
for any ${\bf p}=(p_1,\ldots, p_k)$ with $p_i\in \{0,1\}$. Therefore, ${\bf M}$ is in the polyball ${\bf B_n}(\cM^\perp)$ and has finite rank. Since $\cM^\perp$ is invariant under  ${\bf S}_{i,j}^*\otimes I_\cE$,  we deduce that
\begin{equation*}
\begin{split}
\rank \left[(id-\Phi_{M_1}^{q_1+1})\circ \cdots \circ (id-\Phi_{M_k}^{q_k+1})(I_{\cM^\perp})\right]
&=
\rank \left[P_{\cM^\perp}(id-\Phi_{{\bf S}_1\otimes I}^{q_1+1})\circ \cdots \circ (id-\Phi_{{\bf S}_k\otimes I}^{q_k+1})(I)|_{\cM^\perp}\right]\\
&=\rank \left[P_{\cM^\perp}(P_{\leq(q_1,\ldots, q_k)}\otimes I_\cE) |_{\cM^\perp} \right]\\
&=\rank \left[P_{\cM^\perp}(P_{\leq(q_1,\ldots, q_k)}\otimes I_\cE)  \right]
\end{split}
\end{equation*}
for any $q_i\in \ZZ_+$. Hence and using Theorem \ref{Euler1}, we conclude that    $\chi({\bf M})$ exists and
\begin{equation*}
\begin{split}
\chi({\bf M})&=
 \lim_{q_1\to\infty}\cdots\lim_{q_k\to\infty}
\frac{\rank \left[P_{\cM^\perp}(P_{\leq(q_1,\ldots, q_k)}\otimes I_\cE)  \right]}{\rank \left[P_{\leq(q_1,\ldots, q_k)}\right]}=\chi(\cM^\perp).
\end{split}
\end{equation*}
This  completes the proof.
\end{proof}

In what follows, we prove a multiplicative property for the Euler characteristic.

\begin{theorem} \label{multiplicative}
Given a function $\kappa :\NN\to \NN$ and ${\bf n}^{(i)}\in \NN^{\kappa(i)}$ for  $i\in \{1,\ldots,p\}$, let ${\bf S}^{({\bf n}^{(i)})}$ and ${\bf S}^{({\bf n}^{(1)},\ldots, {\bf n}^{(p)})}$ be the universal models of the polyballs ${\bf B}_{{\bf n}^{(i)}}$ and
${\bf B}_{({\bf n}^{(1)},\ldots, {\bf n}^{(p)})}$, respectively.
For each $i\in \{1,\ldots, p\}$, assume that
\begin{enumerate}
\item[(i)] $\cE_i$ is a finite dimensional Hilbert space;
\item[(ii)] $\cM_i$ is an invariant subspace under ${\bf S}^{({\bf n}^{(i)})}\otimes I_{\cE_i}$.
 \end{enumerate}

Then the Euler characteristic satisfies the equation
 $$
 \chi(\cM_1^\perp\otimes \cdots \otimes \cM_p^\perp)=\prod_{i=1}^p \chi(\cM_i^\perp),
 $$
 where, under the appropriate  identification, $\cM_1^\perp\otimes \cdots \otimes \cM_p^\perp$ is viewed as a coinvariant subspace for ${\bf S}^{({\bf n}^{(1)},\ldots, {\bf n}^{(p)})}\otimes I_{\cE_1\otimes\cdots \otimes \cE_p}$.
\end{theorem}
\begin{proof} For each $i\in\{1,\ldots, p\}$ let ${\bf n}^{(i)}:=(n_1^{(i)},\ldots, n_{\kappa(i)}^{(i)})\in \NN^{\kappa(i)}$.
 Given $j\in \{1,\ldots, \kappa(i)\}$, let $F^2(H_{n_j}^{(i)})$ be the full Fock space with $n_j^{(i)}$ generators, and denote by $P_{q_j^{(i)}}^{(n_j^{(i)})}$
the orthogonal projection  of $F^2(H_{n_j^{(i)}})$ onto the span of all homogeneous polynomials of $F^2(H_{n_j^{(i)}})$ of degree equal to $q_j^{(i)}\in \ZZ_+$.
Consider the orthogonal projections
 $$
 Q_1:=P_{\leq q_1^{(1)}}^{(n_1^{(1)})}\otimes \cdots \otimes P_{\leq q_{\kappa(1)}^{(1)}}^{(n_{\kappa(1)}^{(1)})}, \quad \cdots, \quad
 Q_p:=P_{\leq q_1^{(p)}}^{(n_1^{(p)})}\otimes \cdots \otimes P_{\leq q_{\kappa(p)}^{(p)}}^{(n_{\kappa(p)}^{(p)})}
 $$
and let $U$ be the unitary operator which provides the  canonical  identification of the Hilbert tensor product
$\otimes_{i=1}^p\left[ F^2(H_{n_1}^{(i)})\otimes \cdots \otimes F^2(H_{n_{\kappa(i)}}^{(i)})\otimes \cE_i\right]$ with
$\left\{\otimes_{i=1}^p\left[ F^2(H_{n_1}^{(i)})\otimes \cdots \otimes F^2(H_{n_{\kappa(i)}}^{(i)})\right]\right\}\otimes (\cE_1\otimes \cdots\otimes \cE_p)$. Note that
$$
U[(Q_1\otimes I_{\cE_1})\otimes \cdots \otimes (Q_p\otimes I_{\cE_p})]
=[Q_1\otimes \cdots \otimes Q_p\otimes I_{\cE_1\otimes\cdots\otimes \cE_p}]U
$$
and $U(\cM_1^\perp\otimes \cdots \otimes \cM_p^\perp)$ is a coinvariant subspace under
${\bf S}^{({\bf n}^{(1)},\ldots, {\bf n}^{(p)})}\otimes I_{\cE_1\otimes\cdots \otimes \cE_p}$.
Due to    Theorem \ref{Euler-cha}, we obtain
 \begin{equation*}
 \begin{split}
\chi\left(\otimes_{i=1}^k \cM_i^\perp\right)&=  \lim   \frac{\rank\left[P_{U(\cM_1^\perp\otimes \cdots \otimes \cM_p^\perp)}
 \left(Q_1
 \otimes \cdots \otimes Q_p\otimes I_{\cE_1\otimes\cdots\otimes \cE_p}
 \right)\right]}{\rank\left[Q_1\otimes \cdots \otimes Q_p \right]}\\
 &=
  \lim   \frac{\rank\left[U^*P_{U(\cM_1^\perp\otimes \cdots \otimes \cM_p^\perp)} UU^*
 \left(Q_1
 \otimes \cdots \otimes Q_p\otimes I_{\cE_1\otimes\cdots\otimes \cE_p}
 \right)U\right]}{\rank\left[Q_1\otimes \cdots \otimes Q_p \right]}\\
 &=
 \lim   \frac{\rank\left\{P_{\cM_1^\perp\otimes \cdots \otimes \cM_p^\perp} \left[
 \left(Q_1\otimes I_{\cE_1}\right)
 \otimes \cdots \otimes \left(Q_p\otimes I_{\cE_p}\right)\right]
  \right\}}{\rank\left[Q_1\otimes \cdots \otimes Q_p \right]},
 \end{split}
 \end{equation*}
 where the limit is taken over ${(q_1^{(1)},\ldots, q_{\kappa(1)}^{(1)},\ldots, q_1^{(p)},\ldots, q_{\kappa(p)}^{(p)})\in \ZZ_+^{\kappa(1)+\cdots +\kappa(p)}}$.
 We remark that  the latter limit is equal to the product
$$
 \lim_{(q_1^{(1)},\ldots, q_{\kappa(1)}^{(1)})\in \ZZ_+^{\kappa(1)}}   \frac{\rank\left[P_{\cM_1^\perp}
 \left(Q_1\otimes I_{\cE_1} \right)\right]}
 {\rank\left[Q_1 \right]}\quad
  \cdots \quad  \lim_{(q_1^{(p)},\ldots, q_{\kappa(p)}^{(p)})\in \ZZ_+^{\kappa(p)}}
  \frac{\rank\left[P_{\cM_p^\perp}\left(Q_p \otimes I_{\cE_p}\right)
 \right]}
 {\rank\left[ Q_p\right]},
 $$
which, due to Theorem \ref{Euler-cha}, is equal  to
$\prod_{i=1}^p\chi\left( \cM_i^\perp\right)$. Therefore, we have
$$
\chi\left(\otimes_{i=1}^p \cM_i^\perp\right)=\prod_{i=1}^p\chi\left( \cM_i^\perp\right).
$$
  The proof is complete.
\end{proof}

  As a particular case, we remark that if  $\cM^\perp\subset H^2(\DD^n)\otimes \CC^r$ and $\cN^\perp\subset H^2(\DD^p)\otimes \CC^q$ are coinvariant subspaces, then so is
$\cM^\perp\otimes \cN^\perp\subset H^2(\DD^{n+p})\otimes \CC^{rq}$ and the Euler characteristic  has the multiplicative property
$\chi(\cM^\perp\otimes \cN^\perp)=\chi(\cM^\perp)\chi(\cN^\perp)$.

The next result shows  that there  are  many non-isomorphic pure elements in the polyball with the same Euler characteristic.

\begin{theorem} \label{value} Let ${\bf n}=(n_1,\ldots, n_k)\in \NN^k$ be such that $k\geq 2$ and  $n_i\geq 2$   for all $i\in \{1,\ldots, k\}$. Then, for each $t\in(0,1)$,  there exists an uncountable family $\{T^{(\omega)}(t)\}_{\omega\in \Omega}$ of pure elements in the regular polyball ${\bf B_n}$ with the following properties:
\begin{enumerate}
\item[(i)] $T^{(\omega)}(t)$ in not unitarily equivalent to $T^{(\sigma)}(t)$ for any $\omega, \sigma\in \Omega$, $\omega\neq \sigma$.
\item[(ii)] $\rank [T^{(\omega)}(t)]=1$ and
$\chi [T^{(\omega)}(t)]=t$ for all $\omega\in \Omega.
$
\end{enumerate}
\end{theorem}

\begin{proof}   If $t\in[0,1)$, there exists a subsequence  of natural numbers $\{k_p\}_{p=1}^N$, $1\leq k_1<k_2<\cdots$, where $N\in \NN$ or $N=\infty$,
and $d_p\in\{1,2,\ldots, n_i-1\}$,  such that $1-t=\sum_{p=1}^N \frac{d_p}{n_i^{k_p}}$.   Define the following subsets of $\FF_{n_i}^+$, the unital free semigroup on $n_i$ generators
$g_{1}^i,\ldots, g_{n_i}^i$ and the identity $g_{0}^i$:
\begin{equation*}
\begin{split}
J_1^i&:=\left\{(g_1^i)^{k_1},\ldots, (g_{d_1}^i)^{k_1}\right\},\\
J_p^1&:=\left\{(g_1^i)^{k_p-k_{p-1}}(g_{n_i}^1)^{k_{p-1}}, (g_2^i)^{k_p-k_{p-1}}(g_{n_i}^i)^{k_{p-1}},\ldots, (g_{d_p}^i)^{k_p-k_{p-1}}(g_{n_i}^i)^{k_{p-1}}\right\}, \quad p=2,3,\ldots, N.
\end{split}
\end{equation*}
We remark  that
\begin{equation}
\label{Mi}
\cM_i(t):=\bigoplus_{\beta\in \cup_{p=1}^N J_p^i} F^2(H_{n_i})\otimes e_\beta^i
\end{equation}
is an  invariant  subspace of $F^2(H_{n_i})$.
 Let $k_p\leq q_i<k_{p+1}$ and note that
\begin{equation*}
\begin{split}
\frac{\rank\left[P_{\cM_i(t)^\perp}P_{ q_i}^{(i)}  \right]}{\rank\left[P_{ q_i}^{(i)}\right]}
&=\frac{1}{\rank\left[P_{ q_i}^{(i)} \right]}
\sum_{\beta_i\in\FF_{n_i}^+}
\left<P_{\cM_i(t)^\perp}P_{ q_i}^{(i)}e^i_{\beta_i},
e^i_{\beta_i} \right>\\
&=
1-\sum_{\beta_i\in\FF_{n_i}^+}
\frac{\left<P_{ q_i}^{(i)} P_{\cM_i(t)} P_{ q_i}^{(i)} e_{\beta_i}^i, e_{\beta_i}^i\right>}{n_i^{q_i}}
=1-\frac{1}{n_i^{q_i}}\sum_{\beta_i\in \FF_{n_i}^+, |\beta_i|=q_i}\|P_{\cM_i(t)}e_{\beta_i}^i\|^2
\\
&=1-\frac{d_1 n_i^{q_i-k_1}+\cdots + d_p n_i^{q_i-k_p}}{n_i^{q_i}}
=1- \left(\frac{d_1}{n_i^{k_1}}+\cdots + \frac{d_p}{n_i^{k_p}}\right).
\end{split}
\end{equation*}
Consequently, we have
\begin{equation*}
\begin{split}
\lim_{q_i\to \infty}\frac{\rank\left[P_{\cM_i(t)^\perp}P_{q_i}^{(i)}  \right]}{\rank\left[P_{q_i}^{(i)}\right]}
 = 1-\sum_{p=1}^N \frac{d_p}{n_1^{k_p}}=t.
\end{split}
\end{equation*}
Now, using Theorem \ref{Euler-cha} (when $k=1$) and  the Stoltz-Cesaro limit theorem, we deduce that
\begin{equation*}
\begin{split}
\chi_i[P_{\cM_i(t)^\perp} {\bf S}_i|_{\cM_i(t)^\perp}]
&=\lim_{q_i\to \infty}\frac{\rank\left[P_{\cM_i(t)^\perp}P_{\leq q_i}^{(i)}  \right]}{\rank\left[P_{\leq q_i}^{(i)}\right]}\\
&=\lim_{q_i\to \infty}\frac{\sum_{s_i=0}^{q_i}\rank\left[P_{\cM_i(t)^\perp}P_{s_i}^{(i)}  \right]}{\sum_{s_i=0}^{q_i}\rank\left[P_{s_i}^{(i)}\right]}\\
&=\lim_{q_i\to \infty}\frac{\rank\left[P_{\cM_i(t)^\perp}P_{q_i}^{(i)}  \right]}{\rank\left[P_{q_i}^{(i)}\right]}
 = 1-\sum_{p=1}^N \frac{d_p}{n_1^{k_p}}=t,
\end{split}
\end{equation*}
where $\chi_i$ stands for the Euler characteristic on the polyball ${\bf B}_{n_i}$.
Let $t\in (0,1)$ and $\omega\in (t, 1)$, and
  define the subspace
$$
\cM^{(\omega)}(t):=\left(\cM_1(\omega)^\perp\otimes \cM_2\left(\frac{t}{ \omega}\right)^\perp\otimes F^2(H_{n_3})\otimes \cdots \otimes F^2(H_{n_k})\right)^\perp.
$$
 Note that $\cM^{(\omega)}(t)$ is an invariant subspace under ${\bf S}_{i,j}$ for any $i\in\{1,\ldots, k\}$ and $j\in \{1,\ldots, n_i\}$.
Setting
$T^{(\omega)}(t):=P_{{\cM^{(\omega)}(t)}^\perp} {\bf S}|_{{\cM^{(\omega)}(t)}^\perp}$
and using Theorem \ref{multiplicative}, we deduce that
\begin{equation*}
\begin{split}
\chi [T^{(\omega)}(t)]= \chi_1\left(P_{\cM_1(\omega)^\perp} {\bf S}_1|_{\cM_1(\omega)^\perp}\right) \chi_2\left(P_{\cM_2(\frac{ t}{\omega})^\perp} {\bf S}_2|_{\cM_2(\frac{ t}{\omega})^\perp}\right)
 =t.
\end{split}
\end{equation*}
Let $\sigma\in (0,t)$.
Since  the Euler characteristic is a unitary invariant and
$$\chi_1[P_{\cM_1(\omega)^\perp} {\bf S}_1|_{\cM_1(\omega)^\perp}]={\omega}\neq \sigma=\chi_1[P_{\cM_1(\sigma)^\perp} {\bf S}_1|_{\cM_1(\sigma)^\perp}],
$$
we deduce that $P_{\cM_1(\omega)^\perp} {\bf S}_1|_{\cM_1(\omega)^\perp}$ is not unitary equivalent to
$P_{\cM_1(\sigma)^\perp} {\bf S}_1|_{\cM_1(\sigma)^\perp}$.
According to Theorem 2.10 from \cite{Po-curvature-polyballs}, when $\cE=\CC$, we have
$\cM_1(\omega)\neq \cM_1(\sigma)$, which  implies  $\cM^{(\omega)}(t)\neq \cM^{(\sigma)}(t)$. Using again the above-mentioned result from \cite{Po-curvature-polyballs}, we conclude that $T^{(\omega)}(t)$ in not unitarily equivalent to $T^{(\sigma)}(t)$ for any $\omega, \sigma\in (0,t)$, $\omega\neq \sigma$. The fact that $\rank [T^{(\omega)}(t)]=1$ is obvious.
The proof is complete.
\end{proof}

\begin{corollary} \label{range} Let ${\bf n}=(n_1,\ldots, n_k)\in \NN^k$ be such that  $n_i\geq 2$ and let  $t\in [0,1]$. Then  there exists a pure element  ${\bf T}$  in the polyball ${\bf B_n}$ such that $\rank ({\bf T})=1$ and
$$
\chi({\bf T})=t.
$$
\end{corollary}
\begin{proof}
If $t=1$, the Euler characteristic  $\chi({\bf S})=1$. When $t\in [0,1)$,  we consider the
subspace $\cM(t):=(\cM_1(t)^\perp\otimes F^2(H_{n_2})\otimes \cdots \otimes F^2(H_{n_k}))^\perp$, where $\cM_1(t)$ is defined by relation \eqref{Mi}.
As in the proof of Theorem \ref{value}, one can show that
$\chi (P_{\cM(t)^\perp} {\bf S}|_{\cM(t)^\perp})=t$. This   completes the proof.
\end{proof}

Using Corollary \ref{range} and tensoring  with the identity on $\CC^m$, one can easily see that for any $t\in [0,m]$, there exists a pure element ${\bf T}$  in the polyball ${\bf B_n}$,  with
$\rank ({\bf T})=m$ and $\chi ({\bf T})=t$. Consequently, the range of the Euler characteristic coincides with the interval $[0,\infty)$.

\begin{theorem} Let  ${\bf T}\in {\bf B_n}(\cH)$  have finite rank and let $\cM$ be an invariant subspace under ${\bf T}$ such that  ${\bf T}|_\cM\in {\bf B_n}(\cM)$ and $\dim \cM^\perp<\infty$. Then ${\bf T}|_\cM$ has finite rank and
$$
\left|\chi({\bf T})-\chi({\bf T}|_\cM)\right|
\leq \dim \cM^\perp\prod_{i=1}^k (n_i-1).
$$
\end{theorem}
\begin{proof} First, note that $\rank ({\bf T}|_\cM)=\rank {\bf \Delta_T}(P_\cM)$. Taking into account  the fact that
${\bf \Delta_T}(P_\cM)={\bf \Delta_T}(I_\cH)-{\bf \Delta_T}(P_{\cM^\perp})$, we deduce that $\rank ({\bf T}|_\cM)<\infty$.
On the other hand, since $\cM$ is an invariant subspace under ${\bf T}$, we have
\begin{equation*}
\begin{split}
\rank &\left[\left(id -\Phi_{T_1|_\cM}^{q_1+1}\right)\circ \cdots \circ\left(id -\Phi_{ T_k|_\cM}^{q_k+1}\right)(I_\cM)\right]\\
&= \rank \left[\left(id -\Phi_{T_1}^{q_1+1}\right)\circ \cdots \circ\left(id -\Phi_{ T_k}^{q_k+1}\right)(P_\cM)\right]\\
&\leq \rank\left[\left(id -\Phi_{T_1}^{q_1+1}\right)\circ \cdots \circ\left(id -\Phi_{ T_k}^{q_k+1}\right)(I_\cH)\right]+\rank \left[\left(id -\Phi_{T_1}^{q_1+1}\right)\circ \cdots \circ\left(id -\Phi_{ T_k}^{q_k+1}\right)(P_{\cM^\perp})\right].
\end{split}
\end{equation*}
Since
\begin{equation*}
\begin{split}
\rank &\left[\left(id -\Phi_{T_1}^{q_1+1}\right)\circ \cdots \circ\left(id -\Phi_{ T_k}^{q_k+1}\right)(P_{\cM^\perp})\right]\leq (1+n_1^{q_1+1})\cdots (1+n_k^{q_k+1})
\rank [P_{\cM^\perp}],
\end{split}
\end{equation*}
we deduce that
\begin{equation*}
\begin{split}
&\left|\frac{\rank\left[\left(id -\Phi_{T_1}^{q_1+1}\right)\circ \cdots \circ\left(id -\Phi_{ T_k}^{q_k+1}\right)(I_\cH)\right]}{{\prod_{i=1}^k(1+n_i+\cdots + n_i^{q_i})}}-\frac{\rank \left[\left(id -\Phi_{T_1|_\cM}^{q_1+1}\right)\circ \cdots \circ\left(id -\Phi_{ T_k|_\cM}^{q_k+1}\right)(I_\cM)\right]}{{\prod_{i=1}^k(1+n_i+\cdots + n_i^{q_i})}}
 \right|\\
 &\qquad\qquad\leq \frac{(1+n_1^{q_1+1})\cdots (1+n_k^{q_k+1})}{{\prod_{i=1}^k(1+n_i+\cdots + n_i^{q_i})}}\text{\rm trace}\,[P_{\cM^\perp}].
\end{split}
\end{equation*}
 Since  $n_i\geq 2$, we can use Theorem \ref{Euler1} and  obtain
$$
\left|\chi({\bf T})-\chi({\bf T}|_\cM)\right|
\leq \dim \cM^\perp\prod_{i=1}^k (n_i-1).
$$
The proof is complete.
\end{proof}

\begin{theorem} Let  ${\bf T}\in {\bf B_n}(\cH)$  have finite rank and let $\cM$ be a co-invariant subspace under ${\bf T}$   with $\dim \cM^\perp<\infty$. Then $P_\cM{\bf T}|_\cM$ has finite rank and
$$
\chi({\bf T})=\chi(P_\cM{\bf T}|_\cM).
$$
\end{theorem}
\begin{proof} Denote $A:=(A_1,\ldots, A_k)$ and  $A_i:=(A_{i,1},\ldots, A_{i,n_i})$, where $A_{i,j}:=P_\cM T_{i,j}|_\cM$ for $i\in \{1,\ldots, k\}$ and $j\in \{1,\ldots, n_i\}$.
Using the fact that $\Phi_{A_i}^{q_i}(I_{\cM})=P_{\cM} \Phi_{T_i}^{q_i}(I_\cH)|_\cM$, we deduce that
\begin{equation*}
\begin{split}
&\rank \left[\left(id -\Phi_{A_1}^{q_1+1}\right)\circ \cdots \circ\left(id -\Phi_{ A_k}^{q_k+1}\right)(I_\cM)\right]\\
&\leq \rank \left[\left(id -\Phi_{T_1}^{q_1+1}\right)\circ \cdots \circ\left(id -\Phi_{ T_k}^{q_k+1}\right)(I_\cH)\right]\\
&\leq \rank \left[P_\cM \left(id -\Phi_{T_1}^{q_1+1}\right)\circ \cdots \circ\left(id -\Phi_{ T_k}^{q_k+1}\right)(I_\cH)\right]+
\rank\left[P_{\cM^\perp}\left(id -\Phi_{T_1}^{q_1+1}\right)\circ \cdots \circ\left(id -\Phi_{ T_k}^{q_k+1}\right)(I_\cH)\right]\\
&\leq \rank\left[ \left(id -\Phi_{A_1}^{q_1+1}\right)\circ \cdots \circ\left(id -\Phi_{ A_k}^{q_k+1}\right)(I_\cM)\right]+
\rank \left[P_{\cM^\perp}\left(id -\Phi_{T_1}^{q_1+1}\right)\circ \cdots \circ\left(id -\Phi_{ T_k}^{q_k+1}\right)(I_\cH)\right].
\end{split}
\end{equation*}
Note also that
$$
\rank\left[P_{\cM^\perp}\left(id -\Phi_{T_1}^{q_1+1}\right)\circ \cdots \circ\left(id -\Phi_{ T_k}^{q_k+1}\right)(I_\cH)\right]
\leq \dim\cM^\perp.
$$
Hence, we deduce that
\begin{equation*}
\begin{split}
&\left|\frac{\rank\left[\left(id -\Phi_{T_1}^{q_1+1}\right)\circ \cdots \circ\left(id -\Phi_{ T_k}^{q_k+1}\right)(I_\cH)\right]}{{\prod_{i=1}^k(1+n_i+\cdots + n_i^{q_i})}}-\frac{\rank\left[\left(id -\Phi_{A_1}^{q_1+1}\right)\circ \cdots \circ\left(id -\Phi_{ A_k}^{q_k+1}\right)(I_\cM)\right]}{{\prod_{i=1}^k(1+n_i+\cdots + n_i^{q_i})}}
 \right|\\
 &\qquad\qquad\leq \frac{ 1}{{\prod_{i=1}^k(1+n_i+\cdots + n_i^{q_i})}}\dim\cM^\perp.
\end{split}
\end{equation*}
Now, using Theorem \ref{Euler1}, we conclude that $\chi({\bf T})=\chi({\bf A})$, which completes the proof.
\end{proof}

According to \cite{Po-curvature}, there are invariant subspaces $\cM_1\subset F^2(H_{n_1})$ which do not contain nonzero polynomials in $F^2(H_{n_1})$. This implies that the set
$\{P_{\cM_1^\perp} e^1_{\alpha_1}\}_{\alpha_1\in \FF_{n_1}^+, |\alpha_1|\leq m}$ is linearly independent for any $m\in \NN$. Consequently, the set
$\{(P_{\cM_1^\perp}\otimes I\otimes \cdots \otimes I) (e^1_{\alpha_1}\otimes \cdots \otimes e_{\alpha_k}^k)\}_{\alpha_i\in \FF_{n_i}^+, |\alpha_i|\leq m}$ is linearly independent for each $m\in \NN$. This shows that the subspace
$\cM_1\otimes F^2(H_{n_2})\otimes \cdots \otimes F^2(H_{n_k})$ does not contain nonzero polynomials. On the other hand, according to \cite{Po-charact} (see also \cite{Po-analytic}),  any invariant subspace $\cM_1\subset F^2(H_{n_1})$ is of Beurling type. Therefore,   $\cM_1\otimes F^2(H_{n_2})\otimes \cdots \otimes F^2(H_{n_k})$ is also of Beurling type.

\begin{lemma} \label{prelim} Let $\cM\subset F^2(H_{n_1})\otimes\cdots \otimes  F^2(H_{n_k})$ be an  nonzero  invariant subspace  satisfying either one of the
 following conditions:
 \begin{enumerate}
 \item[(i)] $\cM$  is a Beurling type invariant subspace which does not contain any polynomial;

 \item[(ii)] $\cM=(\cM_1^\perp\otimes \cdots \cM_k^\perp)^\perp$, where $\{0\}\neq \cM_i\subset F^2(H_{n_i})$ are invariant subspaces  which do not contain  nonzero polynomials.
 \end{enumerate}
 Then
 $$
 \text{\rm curv} \, (P_{\cM^\perp}{\bf S}|_{\cM^\perp})<1 \quad  \text{ and } \quad
 \chi(P_{\cM^\perp}{\bf S}|_{\cM^\perp})=1.
 $$
\end{lemma}
\begin{proof} Assume that $\cM\subset F^2(H_{n_1})\otimes\cdots \otimes  F^2(H_{n_k})$ is an  nonzero   invariant subspace which does not contain any polynomial. Then the set
$\{ P_{\cM^\perp} (e^1_{\alpha_1}\otimes \cdots \otimes e_{\alpha_k}^k)\}_{\alpha_i\in \FF_{n_i}^+, |\alpha_i|\leq m}$ is linearly independent for each $m\in \NN$ and, consequently, using Theorem \ref{Euler-cha}, we deduce that
$$
\chi(P_{\cM^\perp}{\bf S}|_{\cM^\perp})=1.
$$
If, in addition, $\cM$ is a Beurling type invariant subspace, then due to \cite{Po-curvature-polyballs} (see  Proposition 3.14), we have
$\text{\rm curv} \, (P_{\cM^\perp}{\bf S}|_{\cM^\perp})<1$.

To prove part (ii), assume that $\cM=(\cM_1^\perp\otimes \cdots \cM_k^\perp)^\perp$, where $\{0\}\neq \cM_i\subset F^2(H_{n_i})$ are invariant subspaces  which do not contain  nonzero polynomials. Due to the first part of the proof, when $k=1$, we have $\chi(P_{\cM_i^\perp}{\bf S}_i|_{\cM_i^\perp})=1$ and $\text{\rm curv} \, (P_{\cM_i^\perp}{\bf S}_i|_{\cM_i^\perp})<1$. Due to Theorem \ref{Euler-cha} and Theorem \ref{multiplicative}, we have
\begin{equation*}
\begin{split}
\chi(P_{\cM^\perp}{\bf S}|_{\cM^\perp})=\chi(\cM_1^\perp\otimes \cdots \cM_k^\perp) =\prod_{i=1}^k \chi(P_{\cM_i^\perp}{\bf S_i}|_{\cM_i^\perp})=1.
\end{split}
\end{equation*}
On the other hand, using \cite{Po-curvature-polyballs} (see Corollary 2.5), we obtain
$$
\text{\rm curv} \, (P_{\cM^\perp}{\bf S}|_{\cM^\perp})=\text{\rm curv} \, (P_{\cM_1^\perp\otimes \cdots \cM_k^\perp}{\bf S}|_{\cM_1^\perp\otimes \cdots \cM_k^\perp})=\prod_{i=1}^k \text{\rm curv} \, (P_{\cM_i^\perp}{\bf S}|_{\cM_i^\perp})<1
$$
This completes the proof.
\end{proof}

The next result shows that the curvature and the Euler characteristic can be far from each each other.

\begin{proposition} For every $\epsilon\in (0,1)$ and $m\in \NN$, there exists $k\in \NN$ and ${\bf T}\in {\bf B}_{\bf n}(\cH)$ with ${\bf n}=(n_1,\ldots, n_k)$ and $n_i\geq 2$  such that
$$0< \text{\rm curv}\, ({\bf T})<\epsilon \quad \text{ and } \quad \chi({\bf T})=\rank ({\bf T})=m.
$$
\end{proposition}
\begin{proof} Due to Lemma \ref{prelim}, there is an invariant subspace
$\cM_1\subseteq F^2(H_{n_1})$  such that
$\chi(P_{\cM_1^\perp}{\bf S}_1|_{\cM_1^\perp})=1$ and $t:=\text{\rm curv} \, (P_{\cM_1^\perp}{\bf S}_1|_{\cM_1^\perp})<1$. Given $\epsilon >0$, take $k\in \NN$ such that $t^k<\epsilon$.
Consider $n_1=\cdots =n_k$ and let $\cM:=(\cM_1^\perp\otimes \cdots \cM_1^\perp)^\perp$.
As in the proof of Lemma \ref{prelim}, we have
$$
\text{\rm curv} \, (P_{\cM^\perp}{\bf S}|_{\cM^\perp}) =t^k<\epsilon \quad \text{ and } \chi(P_{\cM^\perp}{\bf S}|_{\cM^\perp})=1.
$$
  The proof is complete.
\end{proof}

\bigskip

\section{A Gauss-Bonnet-Chern type  theorem on noncommutative polyballs}

In this section, we provide a characterization of the graded invariant subspaces for
the tensor product $F^2(H_{n_1})\otimes\cdots\otimes F^2(H_{n_k})$, and obtain  a version of the Gauss-Bonnet-Chern  theorem on noncommutative polyballs.

 A Hilbert space $\cH$  is called $\ZZ^k$-graded if there is a strongly continuous unitary representation  $U:\TT^k\to B(\cH)$  of the $k$-torus $\TT^k=\TT\times \cdots \times \TT$, where $\TT:=\{z\in \CC:\ |z|=1\}$.
$U$ is called the gauge group of $\cH$.

 \begin{definition}
 A tuple ${\bf T}=(T_1,\ldots, T_k)\in B(\cH)^{n_1}\times\cdots \times B(\cH)^{n_k}$ with $T_i=(T_{i,1},\ldots, T_{i,n_i})$, is called $\ZZ^k$-graded if there is a distinguished gauge group $U$ on $\cH$ such that
$$
U(\lambda)T_{i,j} U(\lambda)^*=\lambda_i T_{i,j}, \qquad \lambda=(\lambda_1,\ldots, \lambda_k)\in \TT^k,
$$
for any $i\in \{1,\ldots, k\}$ and $j\in \{1,\ldots, n_i\}$.
\end{definition}

A closed subspace $\cM\subseteq \cH$ is called graded with respect to $U$ if $U(\lambda)\cM\subseteq \cM$ for all $\lambda\in \TT^k$. In this case, the subrepresentation  $U|_\cM$ is the  gauge group on $\cM$. Similarly, the orthocomplement $\cM^\perp:=\cH\ominus \cM$  is called graded  if $U(\lambda)\cM^\perp\subseteq \cM^\perp$ for all $\lambda\in \TT^k$. Note that $\cM$ is graded if and only if $\cM^\perp$ is graded.

Assume that ${\bf T}$ is graded with respect to the gauge group $U$ on $\cH$ and let $\cM\subseteq \cH$ be  an invariant subspace under ${\bf T}$.  Note that if $\cM$ is  graded with respect to $U$, then ${\bf T}|_\cM$ is graded with respect to the subrepresentation  $U|_\cM$, i.e.
$$
U(\lambda)|_\cM(T_{i,j}|_\cM) =\lambda_i (T_{i,j}|_\cM)U(\lambda)|_\cM, \qquad \lambda=(\lambda_1,\ldots, \lambda_k)\in \TT^k,
$$
for any $i\in \{1,\ldots, k\}$ and $j\in \{1,\ldots, n_i\}$. Similarly, $\cM^\perp$ is graded with respect to $U$ and $P_{\cM^\perp}{\bf T}|_{\cM^\perp}$ is  graded with respect to the subrepresentation  $U|_{\cM^\perp}$.

We remark that the universal model ${\bf S}=({\bf S}_1,\ldots, {\bf S}_k)$,
${\bf S}_i=({\bf S}_{i,1}, \ldots, {\bf S}_{i,j})$, with ${\bf S}_{i,j}$ acting on the tensor  product
$F^2(H_{n_1})\otimes\cdots\otimes F^2(H_{n_k})$, is $\ZZ^k$-graded with respect to the canonical gauge group $U:\TT^k\to B(F^2(H_{n_1})\otimes\cdots\otimes F^2(H_{n_k}))$ defined by
$$
U(\lambda)\left( \sum_{{\alpha_i\in \FF_{n_i}^+}\atop{i\in \{1,\ldots,k\}}}
a_{\alpha_1,\ldots, \alpha_k} e^1_{\alpha_1}\otimes\cdots \otimes   e^k_{\alpha_k}\right)
=
\sum_{{\alpha_i\in \FF_{n_i}^+}\atop{i\in \{1,\ldots,k\}}}
\lambda_1^{|\alpha_1|}\cdots \lambda_k^{|\alpha|}a_{\alpha_1,\ldots, \alpha_k} e^1_{\alpha_1}\otimes \cdots \otimes  e^k_{\alpha_k}.
$$
Given $(s_1,\ldots, s_k)\in \ZZ_+^k$,  we say that a polynomial  in $F^2(H_{n_1})\otimes\cdots\otimes F^2(H_{n_k})$ is multi-homogeneous of order  $(s_1,\ldots, s_k)$ if it has the form
$$q= \sum_{{\alpha_i\in \FF_{n_i}^+, |\alpha_i|=s_i}\atop{i\in \{1,\ldots,k\}}}
a_{\alpha_1,\ldots, \alpha_k} e^1_{\alpha_1}\otimes\cdots \otimes   e^k_{\alpha_k}.
$$
This is equivalent to the fact that
 $U(\lambda)q=\lambda_1^{s_1}\cdots \lambda_k^{s_k} q$  for any
$\lambda=(\lambda_1,\ldots, \lambda_k)\in \TT^k$.

Let   $\cM\subset F^2(H_{n_1})\otimes\cdots\otimes F^2(H_{n_k})$  be an invariant subspace for the universal model ${\bf S}$.
We recall (see Corollary 2.2 from \cite{Po-curvature-polyballs}) that the  defect map ${\bf \Delta_{S}}$ is one-to-one and any  projection $P_\cM$ has the following Taylor type representation around its defect
$$
P_\cM=\sum_{s_1=0}^\infty \Phi_{{\bf S}_1}^{s_1}\left(\sum_{s_2=0}^\infty \Phi_{{\bf S}_2}^{s_2}\left(\cdots \sum_{s_k=0}^\infty \Phi_{{\bf S}_k}^{s_k} \left({\bf \Delta_{S}}(P_\cM)\right)\cdots \right)\right),
$$
where the iterated series converge in the weak operator topology. If, in addition, ${\bf \Delta_{S}}(P_\cM)\geq 0$, then
\begin{equation}
\label{PM}
P_\cM=\sum_{(s_1,\ldots, s_k)\in \ZZ_+^k} \Phi_{{\bf S}_1}^{s_1}\circ \cdots \circ \Phi_{{\bf S}_k}^{s_k}({\bf \Delta_{S}}(P_\cM)).
\end{equation}
We recall that {\it defect operator}  of $\cM$ is given by $\Delta_\cM:={\bf \Delta_{S}}(P_\cM)$. Note that $\Delta_\cM(\cM)\subseteq \cM$ and $\Delta_\cM(\cM^\perp)=\{0\}$.
Moreover, we have ${\bf \Delta_{S|_{\cM}}}(I_\cM)=\Delta_\cM|_\cM.$ Now, it is clear that $\Delta_\cM$ is positive and has finite rank if and only if ${\bf \Delta_{S|_{\cM}}}(I_\cM)$ has the same properties.
The remarks above  show that the invariant subspace $\cM$ is uniquely determined by its defect operator $\Delta_\cM$.

Note  that  any  set $\{\psi_s\}_{s\in \Lambda}$, $\Lambda\subseteq\NN$,  of
multi-homogeneous  polynomials (perhaps of different orders) generates a graded closed invariant subspace  of $F^2(H_{n_1})\otimes\cdots\otimes F^2(H_{n_k})$ by setting
$$\cM:=\overline{\text{\rm span}}\{{\bf S}_{1,\alpha_1}\cdots {\bf S}_{k,\alpha_k}\psi_s:\ \alpha\in \FF_{n_i}^+,  i\in \{1,\ldots, k\}, s\in \Lambda\}.
$$
The next theorem provides a characterization of those graded invariant subspaces $\cM\subseteq F^2(H_{n_1})\otimes\cdots\otimes F^2(H_{n_k})$ with positive defect operator $\Delta_\cM$.

 If ${f}:= \sum_{\beta_i\in \FF_{n_i}^+}a_{\beta_1,\ldots, \beta_k} e^1_{\beta_1}\otimes\cdots \otimes  e^k_{\beta_k}$  is a polynomial in   $F^2(H_{n_1})\otimes\cdots\otimes F^2(H_{n_k})$,  and $X_{i,j}$ are bounded operators on a Hilbert space, where $i\in \{1,\ldots, k\}$ and $j\in \{1,\ldots, n_i\}$,
we define the polynomial calculus
$f(X_{i,j}):=\sum_{\beta_i\in \FF_{n_i}^+}a_{\beta_1,\ldots, \beta_k}X_{1,\beta_1}\cdots X_{k,\beta_k}$.
We use the notation
   $\widetilde f({\bf R}_{i,j}):=\sum_{\beta_i\in \FF_{n_i}^+} a_{\beta_1,\ldots,\beta_k}
   {\bf R}_{1,\widetilde\beta_1} \cdots  {\bf R}_{k,\widetilde
   \beta_k}$, where
$\widetilde \beta_i=g_{j_p}^i\cdots g_{j_1}^i$ denotes the reverse of
$\beta_i=g_{j_1}^i\cdots g_{j_p}^i\in \FF_{n_i}^+$.

\begin{theorem} Let $\cM\subseteq F^2(H_{n_1})\otimes\cdots\otimes F^2(H_{n_k})$ be an invariant subspace. Then  $\cM$ is graded  and has positive  defect operator $\Delta_\cM$ if and only if there is a sequence of  multi-homogeneous  polynomials $\{\psi_s\}_{s=1}^N$, where $N\in \NN$ or $N=\infty$, with the following properties:
\begin{enumerate}
 \item[(i)] each $\psi_s$ is in the range of $\Delta_\cM$;
  \item[(ii)] $\{\widetilde\psi_s({\bf R}_{i,j})\}_{s=1}^N$ are isometries with orthogonal ranges, where $\{{\bf R}_{i,j}\}$ is the right universal model; \item[(iii)] The orthogonal projection $P_\cM$ satisfies the relation
$$
P_\cM=\sum_{s=1}^N \widetilde\psi_s(R_{i,j})\widetilde\psi_s(R_{i,j})^*,$$
 where the convergence is in the strong operator topology.
\end{enumerate}
Moreover, in this case, we have $N=\rank(\Delta_\cM)$
and
$$
  \Delta_\cM\xi=\sum_{t=1}^N\left<\xi, \psi_s\right>\psi_s, \qquad \xi\in \otimes_{i=1}^kF^2(H_{n_i}).
  $$
\end{theorem}
\begin{proof}
Assume that $\cM\subseteq F^2(H_{n_1})\otimes\cdots\otimes F^2(H_{n_k})$ is a graded  invariant subspace  and has positive  defect operator $\Delta_\cM$. Let  $U:\TT^k\to B(F^2(H_{n_1})\otimes\cdots\otimes F^2(H_{n_k}))$ be the canonical gauge group of $F^2(H_{n_1})\otimes\cdots\otimes F^2(H_{n_k})$. Then $U(\lambda_1,\ldots, \lambda_k)\cM=\cM$ for any
$(\lambda_1,\ldots, \lambda_k)\in \TT^k$ and, consequently $U(\lambda_1,\ldots, \lambda_k)$ commutes with $P_\cM$.
Taking into account that  $$
U(\lambda_1,\ldots, \lambda_k){\bf S}_{i,j} U(\lambda_1,\ldots, \lambda_k)^*=\lambda_i {\bf S}_{i,j}, \qquad \lambda=(\lambda_1,\ldots, \lambda_k)\in \TT^k,
$$
for any $i\in \{1,\ldots, k\}$ and $j\in \{1,\ldots, n_i\}$,  one can easily see that
$$
U(\lambda_1,\ldots, \lambda_k)\Delta_\cM=\Delta_\cM U(\lambda_1,\ldots, \lambda_k),  \qquad \lambda=(\lambda_1,\ldots, \lambda_k)\in \TT^k.
$$
For each $(s_1,\ldots, s_k)\in \ZZ_+^k$, let $Q_{(s_1,\ldots, s_k)}$  be the orthogonal projection of $F^2(H_{n_1})\otimes\cdots\otimes F^2(H_{n_k})$ onto the subspace of multi-homogeneous polynomials of order $(s_1,\ldots, s_k)$.
Since
$$
U(\lambda_1,\ldots, \lambda_k)=\sum_{p_1=0}^\infty\cdots \sum_{p_k=0}^\infty
\lambda_1^{\alpha_1}\cdots \lambda_k^{\alpha}Q_{(s_1,\ldots, s_k)},  \qquad \lambda=(\lambda_1,\ldots, \lambda_k)\in \TT^k,
$$
and  due to the spectral theorem, we deduce that
$$
\Delta_\cM Q_{(s_1,\ldots, s_k)}= Q_{(s_1,\ldots, s_k)}\Delta_\cM,\qquad (s_1,\ldots, s_k)\in \ZZ^k.
$$
 If $(p_1,\ldots, p_k)$ is in the set $$\Omega:=\{(s_1,\ldots, s_k)\in \ZZ^k:\ \Delta_\cM Q_{(s_1,\ldots, s_k)}\neq 0\},
  $$
  then $\Delta_\cM Q_{(p_1,\ldots, p_k)}$ is a nonzero  positive finite rank operator  supported in the space of multi-homogeneous polynomials $ Q_{(p_1,\ldots, p_k)}(\otimes_{i=1}^kF^2(H_{n_i}))$. Due to the spectral theorem, it can be expressed as a finite  sum
$\Delta_\cM Q_{(p_1,\ldots, p_k)}=\sum_{t=1}^m \Lambda_t$
 of  rank-one positive operators $\Lambda_t$ defined by
 \begin{equation}
 \label{def-La}
 \Lambda_t\xi:=\left<\xi, \psi_t\right>\psi_t, \qquad \xi\in \otimes_{i=1}^kF^2(H_{n_i}),
  \end{equation}
  where $\{\psi_t\}_{t=1}^m$ are orthonormal  multi-homogeneous polynomials in $ \Delta_\cM Q_{(p_1,\ldots, p_k)}(\otimes_{i=1}^kF^2(H_{n_i}))$.
 Note that
 \begin{equation*}
 \begin{split}
 \widetilde\psi_t({\bf R}_{i,j}) P_\CC \widetilde\psi_t({\bf R}_{i,j})^*\xi=\widetilde\psi_t({\bf R}_{i,j})(\left<\xi,\psi_t\right>)= \left<\xi,\psi_t\right> \psi_t=\Lambda_t \xi.
 \end{split}
 \end{equation*}
 Consequently,  using the fact that $\widetilde\psi_t({\bf R}_{i,j})$ is a multi-analytic operator with respect to the universal model ${\bf S}$,  we have
  \begin{equation} \label{qqqq}
 \sum_{s_1=0}^{q_1}\cdots \sum_{s_k=0}^{q_k} \Phi_{{\bf S}_1}^{s_1} \circ\cdots \circ \Phi_{{\bf S}_k}^{s_k} (\Lambda_t)
 = \widetilde\psi_t({\bf R}_{i,j})\left(\sum_{s_1=0}^{q_1}\cdots \sum_{s_k=0}^{q_k} \Phi_{{\bf S}_1}^{s_1} \circ\cdots \circ \Phi_{{\bf S}_k}^{s_k} ( P_\CC)\right) \widetilde\psi_t({\bf R}_{i,j})^*
 \end{equation}
 Taking the limits as $q_1\to\infty,\ldots, q_k\to\infty$, and using the identity
 $$
I=\sum_{(s_1,\ldots, s_k)\in \ZZ_+^k} \Phi_{{\bf S}_1}^{s_1}\circ \cdots \circ \Phi_{{\bf S}_k}^{s_k}({\bf \Delta_{S}}(I))=\sum_{(s_1,\ldots, s_k)\in \ZZ_+^k} \Phi_{{\bf S}_1}^{s_1}\circ \cdots \circ \Phi_{{\bf S}_k}^{s_k}(P_\CC),
$$
 we deduce that
 \begin{equation}\label{lim}
 \sum_{s_1=0}^{\infty}\cdots \sum_{s_k=0}^{\infty} \Phi_{{\bf S}_1}^{s_1} \circ\cdots \circ \Phi_{{\bf S}_k}^{s_k} (\Lambda_t)=\widetilde\psi_t({\bf R}_{i,j}) \widetilde\psi_t({\bf R}_{i,j})^*.
 \end{equation}
 Now, let $\{\Lambda_t\}_{t=1}^N$, where $N\in \NN$ or $N=\infty$, be the set of the rank-one positive operators associated, as above, with all the operators
 $\Delta_\cM Q_{(p_1,\ldots, p_k)}$  with $(p_1,\ldots, p_k)\in \Omega$, and let $\{\psi_t\}_{t=1}^N$ be the corresponding orthonormal multi-homogeneous polynomials,  according to relation \eqref{def-La}.  We have $\Delta_\cM=\sum_{t=1}^N \Lambda_t$ where the convergence in the strong operator topology, and $\rank \Delta_\cM=N$. Note also that
  $$
  \Delta_\cM\xi=\sum_{t=1}^N\left<\xi, \psi_t\right>\psi_t, \qquad \xi\in \otimes_{i=1}^kF^2(H_{n_i}).
  $$
   Due to the fact each completely positive map $\Phi_{{\bf S}_i}$ is $WOT$-continuous and using relation \eqref{qqqq}, we deduce that
    that
  \begin{equation*}
 \begin{split}
 \sum_{s_1=0}^{q_1}\cdots \sum_{s_k=0}^{q_k} \Phi_{{\bf S}_1}^{s_1} \circ\cdots \circ \Phi_{{\bf S}_k}^{s_k} (\Delta_\cM)
 &=\sum_{t=1}^\infty\left(\sum_{s_1=0}^{q_1}\cdots \sum_{s_k=0}^{q_k} \Phi_{{\bf S}_1}^{s_1} \circ\cdots \circ \Phi_{{\bf S}_k}^{s_k} (\Lambda_t)\right)\\
 &=
 \sum_{t=1}^\infty \widetilde\psi_t({\bf R}_{i,j})\left(\sum_{s_1=0}^{q_1}\cdots \sum_{s_k=0}^{q_k} \Phi_{{\bf S}_1}^{s_1} \circ\cdots \circ \Phi_{{\bf S}_k}^{s_k} ( P_\CC)\right) \widetilde\psi_t({\bf R}_{i,j})^*.
\end{split}
 \end{equation*}
    Due to relation \eqref{PM}, $\{\sum_{s_1=0}^{q_1}\cdots \sum_{s_k=0}^{q_k} \Phi_{{\bf S}_1}^{s_1} \circ\cdots \circ \Phi_{{\bf S}_k}^{s_k} (\Delta_\cM)\}$ is an increasing sequence of positive operators which converges strongly  to $P_\cM$  as $q_1\to\infty,\ldots, q_k\to\infty$, and
$\{ \sum_{s_1=0}^{q_1}\cdots \sum_{s_k=0}^{q_k} \Phi_{{\bf S}_1}^{s_1} \circ\cdots \circ \Phi_{{\bf S}_k}^{s_k} ( P_\CC)\}$ is an increasing sequence of positive operators which converges strongly to the identity operator on
$\otimes_{i=1}^kF^2(H_{n_i})$. Consequently,  the equalities above and relation \eqref{qqqq} can be used to deduce that
$$
P_\cM=\sum_{t=1}^\infty \widetilde\psi_t({\bf R}_{i,j})  \widetilde\psi_t({\bf R}_{i,j})^*,
$$
where the convergence is in the strong operator topology.
Let
    $\Psi:(\otimes_{i=1}^k F^2(H_{n_i}))\otimes
\CC^N \to \otimes_{i=1}^k F^2(H_{n_i})$  be the bounded operator having the $1\times N$ matrix representation
$$
[\widetilde\psi_1({\bf R}_{i,j}), \widetilde\psi_2({\bf R}_{i,j}), \ldots],
$$
  where $\CC^\infty$ stands for $\ell^2(\NN)$.
 Note that $\Psi$ is a multi-analytic operator with respect to the universal model ${\bf S}$ and $\Psi \Psi^*=P_\cM$. Therefore, $\Psi$ is a partial isometry.
Since ${\bf S}_{i,j}$ are isometries, the initial space of $\Psi$, i.e.
$$\Psi^*(\otimes_{i=1}^k F^2(H_{n_i}))
 =\{x\in (\otimes_{i=1}^k F^2(H_{n_i}))\otimes \CC^N: \ \|\Psi x\|=\|x\|\}$$
  is
 reducing under each operator ${\bf S}_{i,j}\otimes I_{\CC^N}$.
 Consequently, since ${\bf \Delta_{S}}(I)=P_\CC$,  to prove that $\Psi$ is an isometry, it is enough to show that
$$\CC= P_\CC\widetilde\psi_t({\bf R}_{i,j})^* (\otimes_{i=1}^k F^2(H_{n_i}))
 \quad \text{ for each }  t\in \{1,\ldots, N\}.
 $$
The latter equality is true due to the fact that
$$
P_\CC\widetilde\psi_t({\bf R}_{i,j})^* (g_t)=\left<\widetilde\psi_t({\bf R}_{i,j})^* (\psi_t),1\right>=\|\psi_t\|^2=1.
$$
 The converse of the theorem is straightforward.
\end{proof}

The next result is a Gauss-Bonnet-Chern type  theorem for rank-one,  graded, and  pure elements  in noncommutative polyballs.
\begin{theorem}\label{GBC} If  ${\bf T}$  is a   graded pure element of rank one  in the noncommutative polyball ${\bf B_n}(\cH)$, with $n_i\geq 2$, then
$$
\text{\rm curv}\, ({\bf T})=\chi({\bf T}).
$$
\end{theorem}
\begin{proof} Let $U:\TT^k\to B(\cH)$ be a strongly continuous unitary representation of the $k$-torus  and assume that ${\bf T}$ is graded with respect to the gauge group $U$.
The spectral subspaces of $U$ are
$$
\cH_{\bf s}=\cH_{(s_1,\ldots, s_k)}:=\left\{ h\in \cH:\ U(\lambda_1,\ldots, \lambda_k)h
=\lambda_1^{s_1}\cdots \lambda_k^{s_k}h \text{ for } (\lambda_1\ldots, \lambda_k)\in \TT^k\right\},
$$
where $ {\bf s}=(s_1,\ldots, s_k)\in \ZZ^k$.    With respect to the orthogonal  decomposition $\cH=\bigoplus_{{\bf s}\in \ZZ^k} \cH_{\bf s}$ we have
$$
U(\lambda_1,\ldots, \lambda_k)=\sum_{(s_1,\ldots, s_k)\in \ZZ^k} \lambda_1^{s_1}\cdots \lambda_k^{s_k} Q_{(s_1,\ldots, s_k)},\qquad
(\lambda_1\ldots, \lambda_k)\in \TT^k,
$$
where $Q_{(s_1,\ldots, s_k)}$ is the orthogonal projection  of $\cH$ onto
$\cH_{(s_1,\ldots, s_k)}$.
Note that if ${\bf T}$ is graded with respect to the  gauge group $U$, then
$$
T_{i,j}\cH_{(s_1,\ldots, s_k)}\subseteq \cH_{(s_1,\ldots,s_{i-1}, s_{i}+1, s_{i+1},\ldots, s_k)},\qquad (s_1,\ldots, s_k)\in \ZZ^k,
$$
for any $i\in \{1,\ldots, k\}$ and $j\in \{1,\ldots, n_i\}$.
Indeed, if $h\in \cH_{(s_1,\ldots, s_k)}$, then $U(\lambda_1,\ldots, \lambda_k)h=\lambda_1^{s_1}\cdots \lambda_k^{s_k} h$ and
\begin{equation*}
\begin{split}
U(\lambda_1,\ldots, \lambda_k)T_{i,j}h = \lambda_i T_{i,j} U(\lambda_1,\ldots, \lambda_k)h
  =\lambda_1^{s_1}\cdots\lambda_{i-1}^{s_{i-1}}\lambda_i^{s_i+1}\lambda_{i+1}^{s_{i+1}}\cdots \lambda_k^{s_k} T_{i,j}h
\end{split}
\end{equation*}
for any $(\lambda_1\ldots, \lambda_k)\in \TT^k$.
Consequently, $T_{i,j}h\in \cH_{(s_1,\ldots,s_{i-1}, s_{i}+1, s_{i+1},\ldots, s_k)}$, which proves our assertion. Therefore,
\begin{equation}\label{TQQT}
T_{i,j} Q_{(s_1,\ldots, s_k)}=Q_{(s_1,\ldots,s_{i-1}, s_{i}+1, s_{i+1},\ldots, s_k)} T_{i,j}
\end{equation}
for any $(s_1,\ldots, s_k)\in \ZZ^k$, $i\in \{1,\ldots, k\}$, and $j\in \{1,\ldots, n_i\}$.
 This also implies
\begin{equation}\label{TQQT2}
 Q_{(s_1,\ldots, s_k)}T_{i,j}^*=T_{i,j}^*Q_{(s_1,\ldots,s_{i-1}, s_{i}+1, s_{i+1},\ldots, s_k)}.
\end{equation}
Consequently, taking into account that $U(\lambda_1,\ldots, \lambda_k)Q_{(s_1,\ldots, s_k)}
=\lambda_1^{s_1}\cdots \lambda_k^{s_k}Q_{(s_1,\ldots, s_k)}$ for any $(\lambda_1,\ldots, \lambda_k)\in \TT^k$, one can easily check that
$${\bf \Delta_{T}}(I)U(\lambda_1,\ldots, \lambda_k)Q_{(s_1,\ldots, s_k)}=U(\lambda_1,\ldots, \lambda_k)Q_{(s_1,\ldots, s_k)}     {\bf \Delta_{T}}(I)
$$
for any $(s_1,\ldots, s_k)\in \ZZ^k$,
which implies
$${\bf \Delta_{T}}(I)^{1/2}U(\lambda_1,\ldots, \lambda_k) =U(\lambda_1,\ldots, \lambda_k) {\bf \Delta_{T}}(I)^{1/2},\qquad (\lambda_1,\ldots, \lambda_k)\in \TT^k.
$$
Due to the spectral theorem, we deduce that
$$
{\bf \Delta_{T}}(I)^{1/2} Q_{(s_1,\ldots, s_k)}= Q_{(s_1,\ldots, s_k)}     {\bf \Delta_{T}}(I)^{1/2},\qquad (s_1,\ldots, s_k)\in \ZZ^k.
$$

Now, we assume that ${\bf T}$ has rank one. Then there exists $(d_1,\ldots, d_k)\in \ZZ_+^k$ such that
\begin{equation}\label{DQDQ}
{\bf \Delta_{T}}(I)^{1/2}=Q_{(d_1,\ldots, d_k)}{\bf \Delta_{T}}(I)^{1/2}={\bf \Delta_{T}}(I)^{1/2}Q_{(d_1,\ldots, d_k)}\neq 0
\end{equation}
and ${\bf \Delta_{T}}(I)^{1/2} Q_{(s_1,\ldots, s_k)}=0$ for any $(s_1,\ldots, s_k)\in \ZZ^k$ with $(s_1,\ldots, s_k)\neq (d_1,\ldots, d_k)$. Consequently,
\begin{equation*}
{\bf \Delta_{T}}(I)^{1/2}(\cH)={\bf \Delta_{T}}(I)^{1/2}\cH_{(d_1,\ldots, d_k)} \subseteq \bigoplus_{s_1\geq d_1,\ldots, s_k\geq d_k} \cH_{(s_1, \ldots, s_k)}.
\end{equation*}
Since $\bigoplus_{s_1\geq d_1,\ldots, s_k\geq d_k} \cH_{(s_1, \ldots, s_k)}$ is invariant under each operator $T_{i,j}$ for  any $i\in \{1,\ldots, k\}$ and $j\in \{1,\ldots, n_i\}$, we have
\begin{equation} \label{inclu}
T_{1,\beta_1}\cdots T_{k,\beta_k}{\bf \Delta_{T}}(I)^{1/2}\cH\subseteq  \bigoplus_{s_1\geq d_1,\ldots, s_k\geq d_k} \cH_{(s_1, \ldots, s_k)}
\end{equation}
for any $\beta_i\in \FF_{n_i}^+$, $ i=1,\ldots,k$.
On the other hand, the  noncommutative Berezin kernel
   $${\bf K_{T}}: \cH \to F^2(H_{n_1})\otimes \cdots \otimes  F^2(H_{n_k}) \otimes  \overline{{\bf \Delta_{T}}(I) (\cH)}$$
   satisfies the relation
\begin{equation}
\label{KDD}
 {\bf K_{T}^*} (e^1_{\beta_1}\otimes \cdots \otimes  e^k_{\beta_k}\otimes  h)=T_{1,\beta_1}\cdots T_{k,\beta_k}{\bf \Delta_{T}}(I)^{1/2}h
 \end{equation}
 for any $\beta_i\in \FF_{n_i}^+$, $ i=1,\ldots,k$,  and $h\in \cH$.
Since  ${\bf T}$ is a pure element in the polyball, the Berezin kernel ${\bf K_{T}}$ is an isometry and, therefore,
$$
\cH=\overline{\text{\rm range} \,{\bf K_{T}^*}}=
\overline{\text{\rm span}}\{ T_{1,\beta_1}\cdots T_{k,\beta_k}{\bf \Delta_{T}}(I)^{1/2}h:\ \beta_i\in \FF_{n_i}^+, h\in \cH\}.
$$
Hence, and using relation \eqref{inclu},  we deduce that
$$
\cH=\bigoplus_{s_1\geq d_1,\ldots, s_k\geq d_k} \cH_{(s_1, \ldots, s_k)}.
$$
On the other hand, taking into account relations \eqref{TQQT}, \eqref{TQQT2}, and \eqref{DQDQ},  we have
\begin{equation*}
\begin{split}
&{\bf \Delta_{T}}(I)^{1/2}T_{k,\alpha_k}^*\cdots T_{1,\alpha_k}^* T_{1,\beta_1}\cdots T_{k,\beta_k}{\bf \Delta_{T}}(I)^{1/2}Q_{(d_1,\ldots, d_k)}\\
&= {\bf \Delta_{T}}(I)^{1/2}T_{k,\alpha_k}^*\cdots T_{1,\alpha_k}^* Q_{(d_1+|\beta_1|,\ldots, d_k+|\beta_k|)}T_{1,\beta_1}\cdots T_{k,\beta_k}{\bf \Delta_{T}}(I)^{1/2}\\
&= {\bf \Delta_{T}}(I)^{1/2}Q_{(d_1+|\beta_1|-|\alpha_1|,\ldots, d_k+|\beta_k|-|\alpha_k|)}T_{k,\alpha_k}^*\cdots T_{1,\alpha_k}^* T_{1,\beta_1}\cdots T_{k,\beta_k}{\bf \Delta_{T}}(I)^{1/2}\\
&=\begin{cases}
Q_{(d_1,\ldots, d_k )}{\bf \Delta_{T}}(I)^{1/2}T_{k,\alpha_k}^*\cdots T_{1,\alpha_k}^* T_{1,\beta_1}\cdots T_{k,\beta_k}{\bf \Delta_{T}}(I)^{1/2}& \ \text{\rm if } |\beta_1|=|\alpha_1|, \ldots, |\beta_k|=|\alpha_k|\\
0 & \  \text{\rm otherwise}
\end{cases}
\end{split}
\end{equation*}
for any $\alpha_1,\beta_1\in \FF_{n_1}^+, \ldots,  \alpha_k,\beta_k\in \FF_{n_k}^+$.
Consequently, we have
\begin{equation}\label{DTD}
{\bf \Delta_{T}}(I)^{1/2}T_{k,\alpha_k}^*\cdots T_{1,\alpha_k}^* T_{1,\beta_1}\cdots T_{k,\beta_k}{\bf \Delta_{T}}(I)^{1/2}=0
\end{equation}
for any $\alpha_1,\beta_1\in \FF_{n_1}^+, \ldots, \alpha_k,\beta_k\in \FF_{n_k}^+$, except when $|\beta_1|=|\alpha_1|, \ldots, |\beta_k|=|\alpha_k|$.
Note that due to relations \eqref{KDD} and \eqref{DTD}, for any $(q_1,\ldots, q_k)\in \ZZ_+^k$, we have
\begin{equation*}
\begin{split}
&(P_{q_1}^{(1)}\otimes \cdots \otimes P_{q_k}^{(k)}\otimes I_\cH){\bf K_{T}}
{\bf K_{T}^*} (e^1_{\beta_1}\otimes \cdots \otimes  e^k_{\beta_k}\otimes h)\\
&=\sum_{\alpha_1\in \FF_{n_1}^+, |\alpha_1|=q_1}\cdots \sum_{\alpha_k\in \FF_{n_k}^+, |\alpha_k|=q_k}
e^1_{\alpha_1}\otimes \cdots \otimes  e^k_{\alpha_k}\otimes{\bf \Delta_{T}}(I)^{1/2}T_{k,\alpha_k}^*\cdots T_{1,\alpha_k}^* T_{1,\beta_1}\cdots T_{k,\beta_k}{\bf \Delta_{T}}(I)^{1/2}h
 \end{split}
\end{equation*}
if $|\beta_1|=q_1, \ldots, |\beta_k|=q_k$,  and zero otherwise.
Now, one can see that
\begin{equation*}
\begin{split}
{\bf K_{T}}
{\bf K_{T}^*} (P_{q_1}^{(1)}\otimes \cdots \otimes P_{q_k}^{(k)}\otimes I_\cH) =(P_{q_1}^{(1)}\otimes \cdots \otimes P_{q_k}^{(k)}\otimes I_\cH){\bf K_{T}}
{\bf K_{T}^*}.
\end{split}
\end{equation*}
for any $(q_1,\ldots, q_k)\in \ZZ_+^k$, where $P_{q_i}^{(i)}$ is the orthogonal projection of the full Fock space $F^2(H_{n_i})$ onto the span of all vectors $e_{\alpha_i}^i$ with  $\alpha\in  \FF_{n_i}^+$ and $|\alpha_i|=q_i$. Since ${\bf K_{T}}
{\bf K_{T}^*}$ is a projection, we deduce that
${\bf K_{T}}
{\bf K_{T}^*} (P_{q_1}^{(1)}\otimes \cdots \otimes P_{q_k}^{(k)}\otimes I_\cH)\}_{(q_1,\ldots, q_k)\in \ZZ_+^k}$ is a net of orthogonal projections with orthogonal ranges. This implies that
$$
\sum_{{(s_1,\ldots, s_k)\in \ZZ_+^k}\atop{(s_1,\ldots, s_k)\leq (q_1,\ldots, q_k)}} {\bf K_{T}}
{\bf K_{T}^*} (P_{q_1}^{(1)}\otimes \cdots \otimes P_{q_k}^{(k)}\otimes I_\cH)= {\bf K_{T}}
{\bf K_{T}^*}(P_{\leq (q_1,\ldots, q_k)}\otimes I)
$$
is an orthogonal projection for any $(q_1,\ldots, q_k)\in \ZZ_+^k$.
Consequently, we have
$$
\text{\rm trace}\,[{\bf K_{T}}{\bf K_{T}^*}(P_{\leq (q_1,\ldots, q_k)}\otimes I)]
=\rank [{\bf K_{T}}{\bf K_{T}^*}(P_{\leq (q_1,\ldots, q_k)}\otimes I)].
$$
Applying Theorem \ref{Euler1}  and  \cite{Po-curvature-polyballs} (see Theorem 1.3), and using the fact that ${\bf K_T}$ is an isometry, we deduce that
\begin{equation*}
\begin{split}
\text{\rm curv}\,({\bf T})&=\lim_{q_1\to\infty}\cdots\lim_{q_k\to\infty}
\frac{\text{\rm trace}\,\left[ {\bf K_{T}^*}(P_{\leq(q_1,\ldots, q_k)}  \otimes I){\bf K_{T}}\right]}{\text{\rm trace}\,\left[P_{\leq(q_1,\ldots, q_k)}\right]}
=\lim_{q_1\to\infty}\cdots\lim_{q_k\to\infty}
\frac{\text{\rm trace}\,\left[{\bf K_{T}} {\bf K_{T}^*}(P_{\leq(q_1,\ldots, q_k)}  \otimes I)\right]}{\text{\rm trace}\,\left[P_{\leq(q_1,\ldots, q_k)}\right]}
\\
&=\lim_{q_1\to\infty}\cdots\lim_{q_k\to\infty}
\frac{\rank\left[{\bf K_{T}} {\bf K_{T}^*}(P_{\leq(q_1,\ldots, q_k)}  \otimes I) \right]}{\rank\left[P_{\leq(q_1,\ldots, q_k)}\right]}
=\lim_{q_1\to\infty}\cdots\lim_{q_k\to\infty}
\frac{\rank\left[ {\bf K_{T}^*}(P_{\leq(q_1,\ldots, q_k)}  \otimes I){\bf K_{T}}\right]}{\rank\left[P_{\leq(q_1,\ldots, q_k)}\right]}\\
&=\chi({\bf T}).
\end{split}
\end{equation*}
The proof is complete.
\end{proof}

\begin{theorem}
If $\cM$ is a graded  invariant subspace of the tensor product $F^2(H_{n_1})\otimes\cdots\otimes F^2(H_{n_k})$,
then
\begin{equation*}
\begin{split}
\lim_{q_1\to\infty}\cdots\lim_{q_k\to\infty}
\frac{\rank\left[P_{\cM^\perp}P_{\leq(q_1,\ldots, q_k)}  \right]}{\rank\left[P_{\leq(q_1,\ldots, q_k)}\right]}
=\lim_{q_1\to\infty}\cdots\lim_{q_k\to\infty}
\frac{\text{\rm trace}\,\left[P_{\cM^\perp}P_{\leq(q_1,\ldots, q_k)}  \right]}{\text{\rm trace}\, \left[P_{\leq(q_1,\ldots, q_k)}\right]}.
\end{split}
\end{equation*}
This is equivalent to
$$\chi(P_{\cM^\perp}{\bf S}|_{\cM^\perp})
=\text{\rm curv}\,(P_{\cM^\perp}{\bf S}|_{\cM^\perp}).
$$
\end{theorem}
\begin{proof}
Combining Theorem \ref{Euler-cha} with Theorem \ref{GBC}, and using Theorem 3.1 from \cite{Po-curvature-polyballs}, the result follows.
\end{proof}

A closer look at the proof of Theorem \ref{GBC}, reveals that if ${\bf T}\in {\bf B_n}(\cH)$ is   a  pure element of finite rank  and
\begin{equation*}
{\bf \Delta_{T}}(I)^{1/2}T_{k,\alpha_k}^*\cdots T_{1,\alpha_k}^* T_{1,\beta_1}\cdots T_{k,\beta_k}{\bf \Delta_{T}}(I)^{1/2}=0
\end{equation*}
for any $\alpha_1,\beta_1\in \FF_{n_1}^+, \ldots, \alpha_k,\beta_k\in \FF_{n_k}^+$, except when $|\beta_1|=|\alpha_1|, \ldots, |\beta_k|=|\alpha_k|$, then
$$
\text{\rm curv}\, ({\bf T})=\chi({\bf T}).
$$

Our final result is a version of the Gauss-Bonnet-Chern   theorem for graded pure elements with finite rank in the noncommutative polyball.
\begin{theorem}\label{GBC2} Let ${\bf T}\in {\bf B_n}(\cH)$, with $n_i\geq 2$,  be  a  pure element of finite rank which is graded with respect to a gauge group $U$, and let $$\cH=\bigoplus_{{\bf s}\in \ZZ^k} \cH_{\bf s}$$ be
the  corresponding orthogonal  decomposition.  Then the spectral subspaces  $\cH_{\bf s}$  of $U$ are finite dimensional  and there exist  ${\bf c},{\bf d}\in \ZZ^k$  with ${\bf c}\leq {\bf d}$ such that
 $\cH=\bigoplus_{{\bf s}\in \ZZ^k, {\bf s}\geq {\bf c}} \cH_{\bf s}$ and
   $$\cH_0:=\bigoplus_{{\bf s}\in \ZZ^k, {\bf s}\geq {\bf d}} \cH_{\bf s}$$
  is an invariant subspace   of ${\bf T}$  with the property that ${\bf T}|_{\cH_0}$  is a finite rank element in ${\bf B_n}(\cH_0)$ and
$$
\text{\rm curv}\, ({\bf T}|_{\cH_0})=\chi({\bf T}|_{\cH_0}).
$$
\end{theorem}
\begin{proof}

Assume that ${\bf T}$ is graded with respect to the gauge group
$$
U(\lambda_1,\ldots, \lambda_k)=\sum_{(s_1,\ldots, s_k)\in \ZZ^k} \lambda_1^{s_1}\cdots \lambda_k^{s_k} Q_{(s_1,\ldots, s_k)},\qquad
(\lambda_1\ldots, \lambda_k)\in \TT^k.
$$
The first part of the proof is similar to that of Theorem \ref{GBC}. Therefore, we can obtain that
$$
{\bf \Delta_{T}}(I)^{1/2} Q_{(s_1,\ldots, s_k)}= Q_{(s_1,\ldots, s_k)}     {\bf \Delta_{T}}(I)^{1/2},\qquad (s_1,\ldots, s_k)\in \ZZ^k.
$$
Since
${\bf \Delta_{T}}(I)^{1/2}$  is a positive finite rank operator  in the commutatnt of $\{Q_{(s_1,\ldots, s_k)}\}_{(s_1,\ldots, s_k)\in \ZZ^k}$,  we must have that
${\bf \Delta_{T}}(I)^{1/2} Q_{(s_1,\ldots, s_k)}=0$ for all but finitely many  $ (s_1,\ldots, s_k)\in \ZZ^k$. Consequently, there are ${\bf c}=(c_1,\ldots, c_k)\leq {\bf d}=(d_1,\ldots, d_k)$ in $\ZZ^k$ such that
\begin{equation} \label{decomp}
{\bf \Delta_{T}}(I)^{1/2}=\sum_{c_1\leq s_1\leq d_1, \ldots, c_k\leq s_k\leq d_k}
D_{(s_1,\ldots, s_k)}^{1/2},
\end{equation}
where $D_{(s_1,\ldots, s_k)}^{1/2}:={\bf \Delta_{T}}(I)^{1/2}Q_{(s_1,\ldots, s_k)}$.
 Taking into account that $\bigoplus_{s_1\geq c_1,\ldots, s_k\geq c_k} \cH_{(s_1, \ldots, s_k)}$ is invariant under each operator $T_{i,j}$ for  any $i\in \{1,\ldots, k\}$ and $j\in \{1,\ldots, n_i\}$, we have
\begin{equation} \label{inclu2}
T_{1,\beta_1}\cdots T_{k,\beta_k}{\bf \Delta_{T}}(I)^{1/2}\cH\subseteq  \bigoplus_{s_1\geq c_1,\ldots, s_k\geq c_k} \cH_{(s_1, \ldots, s_k)}
\end{equation}
for any $\beta_i\in \FF_{n_i}^+$ and  $ i\in \{1,\ldots,k\}$.
Since  ${\bf T}$ is a pure element in the polyball, the Berezin kernel ${\bf K_{T}}$ is an isometry and
\begin{equation}\label{range2}
\cH=\overline{\text{\rm range} \,{\bf K_{T}^*}}=
\overline{\text{\rm span}}\{ T_{1,\beta_1}\cdots T_{k,\beta_k}{\bf \Delta_{T}}(I)^{1/2}h:\ \beta_i\in \FF_{n_i}^+, h\in \cH\}.
\end{equation}
Hence, and using relation \eqref{inclu2},  we deduce that
$
\cH=\bigoplus_{s_1\geq c_1,\ldots, s_k\geq c_k} \cH_{(s_1, \ldots, s_k)}.
$
Due to relation \eqref{decomp}  and the fact  that
$
{\bf \Delta_{T}}(I)^{1/2}
$
has  finite rank, we have
$$
{\bf \Delta_{T}}(I)^{1/2} \cH\subseteq \bigoplus_{c_1\leq s_1\leq d_1, \ldots, c_k\leq s_k\leq d_k} \cM_{(s_1, \ldots, s_k)}\subseteq \bigoplus_{c_1\leq s_1\leq d_1, \ldots, c_k\leq s_k\leq d_k} \cH_{(s_1, \ldots, s_k)},
$$
where $\cM_{(s_1, \ldots, s_k)}:={\bf \Delta_{T}}(I)^{1/2} \cH_{(s_1, \ldots, s_k)}$ is finite dimensional.
On the other hand,
if $\beta_i\in \FF_{n_i}^+$ with $|\beta_i|=m_i\in \ZZ_+$ and $i\in \{1,\ldots, k\}$, we have
\begin{equation*}
T_{1,\beta_1}\cdots T_{k,\beta_k}{\bf \Delta_{T}}(I)^{1/2}\cH\subseteq  \bigoplus_{c_1+m_1\leq s_1\leq d_1+m_1, \ldots, c_k+m_k\leq s_k\leq d_k+m_k} \cH_{(s_1, \ldots, s_k)}
\end{equation*}
and $T_{1,\beta_1}\cdots T_{k,\beta_k}{\bf \Delta_{T}}(I)^{1/2}\cH$ is a finite dimensional subspace. Consequently, using relation \eqref{range2} and the fact that
$
\cH=\bigoplus_{s_1\geq c_1,\ldots, s_k\geq c_k} \cH_{(s_1, \ldots, s_k)},
$
we conclude that all the spectral spaces  $\cH_{(s_1, \ldots, s_k)}$ are finite dimensional.

Since ${\bf T}$ is a pure element in ${\bf B_n}(\cH)$, we have
$$
I=\sum_{p_1=0}^\infty\cdots \sum_{p_k=0}^\infty \Phi_{T_1}^{p_1}\circ \cdots \circ \Phi_{T_k}^{p_k}({\bf \Delta_{T}}(I)),
$$
where the convergence is in the weak operator topology. Hence, after multiplying to the left by $Q_{(q_1,\ldots,q_k)}$,  we obtain
\begin{equation} \label{Qq}
Q_{(q_1,\ldots, q_k)}=\sum_{p_1=0}^\infty\cdots \sum_{p_k=0}^\infty Q_{(q_1,\ldots, q_k)}\Phi_{T_1}^{p_1}\circ \cdots \circ \Phi_{T_k}^{p_k}({\bf \Delta_{T}}(I))
\end{equation}
for any $(q_1,\ldots, q_k)\in \ZZ^k$. On the other hand, taking into account relation \eqref{decomp},
we deduce that
\begin{equation}
\label{Qq2}
Q_{(q_1,\ldots, q_k)}\Phi_{T_1}^{p_1}\circ \cdots \circ \Phi_{T_k}^{p_k}({\bf \Delta_{T}}(I))
=\sum_{s_1=c_1}^{d_1}\cdots \sum_{s_k=c_k}^{d_k}
Q_{(q_1,\ldots, q_k)}\Phi_{T_1}^{p_1}\circ \cdots \circ \Phi_{T_k}^{p_k}(D_{(s_1,\ldots, s_k)})
\end{equation}
for any $(p_1,\ldots, p_k)\in \ZZ_+^k$.
Since ${\bf \Delta_{T}}(I)\leq I$, it is clear that
$D_{(s_1,\ldots, s_k)}\leq Q_{(s_1,\ldots, s_k)}$.  Taking into account that
\begin{equation*}
\begin{split}
T_{i,j} Q_{(s_1,\ldots, s_k)}=Q_{(s_1,\ldots,s_{i-1}, s_{i}+1, s_{i+1},\ldots, s_k)} T_{i,j}
\end{split}
\end{equation*}
for any $(s_1,\ldots, s_k)\in \ZZ^k$, $i\in \{1,\ldots, k\}$, and $j\in \{1,\ldots, n_i\}$,  we have
\begin{equation*}
\begin{split}
\Phi_{T_i}(Q_{(s_1,\ldots, s_k)})&=\sum_{j=1}^{n_i} T_{i,j}Q_{(s_1,\ldots, s_k)}T_{i,j}^*\\
&=Q_{(s_1,\ldots,s_{i-1}, s_{i}+1, s_{i+1},\ldots, s_k)}\Phi_{T_i}(I)Q_{(s_1,\ldots,s_{i-1}, s_{i}+1, s_{i+1},\ldots, s_k)}\\
&\leq Q_{(s_1,\ldots,s_{i-1}, s_{i}+1, s_{i+1},\ldots, s_k)}
\end{split}
\end{equation*}
for any $ i\in\{1,\ldots, k\}$.
Consequently, and using the inequality $D_{(s_1,\ldots, s_k)}\leq Q_{(s_1,\ldots, s_k)}$, we deduce that
\begin{equation}
\label{FIFI}
\Phi_{T_1}^{p_1}\circ \cdots \circ \Phi_{T_k}^{p_k}(D_{(s_1,\ldots, s_k)})\leq
\Phi_{T_1}^{p_1}\circ \cdots \circ \Phi_{T_k}^{p_k}(Q_{(s_1,\ldots, s_k)})\leq
Q_{(s_1+p_1,\ldots, s_k+p_k)}
\end{equation}
for any $(p_1,\ldots, p_k)\in \ZZ_+^k.$
Now, let  $(q_1,\ldots, q_k)\in \ZZ^k$ be  such that  $(q_1,\ldots, q_k)\geq (d_1,\ldots, d_k)$, Since $\{Q_{(s_1,\ldots, s_k)}\}$ are orthogonal projections, relation  \eqref{FIFI} implies
\begin{equation*}
\begin{split}
\sum_{p_1=0}^\infty\cdots& \sum_{p_k=0}^\infty
\sum_{s_1=c_1}^{d_1}\cdots \sum_{s_k=c_k}^{d_k}
Q_{(q_1,\ldots, q_k)}\Phi_{T_1}^{p_1}\circ \cdots \circ \Phi_{T_k}^{p_k}(D_{(s_1,\ldots, s_k)})\\
&=\sum_{s_1=c_1}^{d_1}\cdots \sum_{s_k=c_k}^{d_k}
\Phi_{T_1}^{q_1-s_1}\circ \cdots \circ \Phi_{T_k}^{q_k-s_k}(D_{(s_1,\ldots, s_k)})\\
&=
\Phi_{T_1}^{q_1-d_1}\circ \cdots \circ \Phi_{T_k}^{q_k-d_k}
\left(\sum_{s_1=c_1}^{d_1}\cdots \sum_{s_k=c_k}^{d_k}
\Phi_{T_1}^{d_1-s_1}\circ \cdots \circ \Phi_{T_k}^{d_k-s_k}(D_{(s_1,\ldots, s_k)})\right).
\end{split}
\end{equation*}
Due to relations \eqref{Qq}  and  \eqref{Qq2}, we obtain
\begin{equation}
\label{QFI}
Q_{(q_1,\ldots, q_k)}=\Phi_{T_1}^{q_1-d_1}\circ \cdots \circ \Phi_{T_k}^{q_k-d_k}(B),
\end{equation}
where $$B:=\sum_{s_1=c_1}^{d_1}\cdots \sum_{s_k=c_k}^{d_k}
\Phi_{T_1}^{d_1-s_1}\circ \cdots \circ \Phi_{T_k}^{d_k-s_k}(D_{(s_1,\ldots, s_k)}).
$$
Due to relation \eqref{QFI}, we deduce that
$$
\Phi_{T_i}\left(Q_{(q_1,\ldots, q_k)}\right)=
Q_{(q_1,\ldots,q_{i-1}, q_{i}+1, q_{i+1},\ldots, q_k)}
$$
for any $(q_1,\ldots, q_k)\geq (d_1,\ldots, d_k)$ and $i\in \{1,\ldots, k\}$.
Consider the subspace
$$\cH_0:=\bigoplus_{s_1\geq d_1,\ldots, s_k\geq d_k} Q_{(s_1,\ldots, s_k)}\cH
$$
and  let ${\bf T}|_{\cH_0}:=(T_1|_{\cH_0},\ldots T_k|_{\cH_0})$ with
$T_i|_{\cH_0}:=(T_{i,1}|_{\cH_0},\ldots, T_{i,n_i}|_{\cH_0})$ for $i\in \{1,\ldots, k\}$.
  An inductive argument shows that, for any $(m_1,\ldots, m_k)\in \ZZ_+^k$,
\begin{equation*}
\begin{split}
&(id-\Phi_{T_1|_{\cH_0}}^{m_1+1})\circ\cdots \circ (id-\Phi_{T_k|_{\cH_0}}^{m_k+1})(I_{\cH_0})\\
&=
(id-\Phi_{T_2|_{\cH_0}}^{m_2+1})\circ\cdots \circ (id-\Phi_{T_k|_{\cH_0}}^{m_k+1})
\left(\sum_{d_1\leq s_1\leq d_1+m_1}\sum_{s_2=d_2}^\infty\cdots \sum_{s_k=d_k}Q_{(s_1,s_2,\ldots, s_k)}|_{\cH_0}\right) \\
&=\cdots=(id-\Phi_{T_k|_{\cH_0}}^{m_k+1})
\left(\sum_{d_1\leq s_1\leq d_1+m_1}\cdots  \sum_{d_{k-1}\leq s_{k-1}\leq d_{k-1}+m_{k-1}} \sum_{s_k=d_k}Q_{(s_1,s_2,\ldots,s_{k-1}, s_k)}|_{\cH_0}\right)\\
&= \sum_{d_1\leq s_1\leq d_1+m_1}\cdots  \sum_{d_{k}\leq s_{k}\leq d_{k}+m_{k}}  Q_{(s_1, \ldots,  s_k)}|_{\cH_0}.
\end{split}
\end{equation*}
In particular,  for $m_1=\cdots=m_k=1$, we have
\begin{equation*}
 (id-\Phi_{T_1|_{\cH_0}})\circ\cdots \circ (id-\Phi_{T_k|_{\cH_0}})(I_{\cH_0}) =Q_{(d_1,d_2,\ldots,d_{k})}|_{\cH_0}\geq 0,
 \end{equation*}
which shows  that ${\bf T}|_{\cH_0}\in {\bf B_n}(\cH_0)$ and has finite rank.
On the other hand, due to the relations above,  the operator $(id-\Phi_{T_1|_{\cH_0}}^{m_1+1})\circ\cdots \circ (id-\Phi_{T_k|_{\cH_0}}^{m_k+1})(I_{\cH_0})$ is an orthogonal projection
for any $(m_1,\ldots, m_k)\in \ZZ_+^k$. Therefore,
$$
\frac{\text{\rm trace}\, \left[(id-\Phi_{T_1}^{m_1+1})\circ \cdots \circ (id-\Phi_{T_k}^{m_k+1})(I)\right]}
{\prod_{i=1}^k(1+n_i+\cdots + n_i^{m_i})}=\frac{\rank \left[(id-\Phi_{T_1}^{m_1+1})\circ \cdots \circ (id-\Phi_{T_k}^{m_k+1})(I)\right]}
{\prod_{i=1}^k(1+n_i+\cdots + n_i^{m_i})}.
$$
 Consequently, applying   Theorem \ref{Euler1} and  using \cite{Po-curvature-polyballs} (see Theorem 1.3 and Corollary 1.4),  we deduce  that
$
\text{\rm curv}\, ({\bf T}|_{\cH_0})=\chi({\bf T}|_{\cH_0}).
$
This completes  the proof.
\end{proof}

       %

      \end{document}